\documentclass[10pt]{article}   
\usepackage[T1]{fontenc}
\usepackage{amssymb,amsmath, amsthm, amsfonts}
\usepackage{graphicx}
\usepackage{listings}
\usepackage{lstautogobble}
\usepackage{enumerate}
\usepackage[shortlabels]{enumitem}
\usepackage{thmtools}
\usepackage{thm-restate}
\usepackage{amsthm}
\usepackage{mathtools}
\usepackage{physics}
\usepackage{color}
\usepackage{hyperref}
\usepackage{authblk}
\usepackage[capitalise]{cleveref}
\usepackage[bottom]{footmisc}
\usepackage{mathrsfs}
\usepackage[final]{showlabels}
\setlist{  
  listparindent=\parindent,
  parsep=0pt,
}
\theoremstyle{plain}
\newtheorem{thm}{Theorem}[section]
\newtheorem{prop}[thm]{Proposition}
\newtheorem{lemma}[thm]{Lemma}
\newtheorem{cor}[thm]{Corollary}

\theoremstyle{definition}
\newtheorem{mydef}[thm]{Definition}

\newtheorem{remark}[thm]{Remark}

\numberwithin{equation}{section} 

\DeclarePairedDelimiter{\paren}{\lparen}{\rparen}

\DeclareMathOperator{\supp}{supp}

\DeclareMathOperator{\dist}{dist}
\DeclareMathOperator{\sgn}{sgn}

\renewcommand{\PV}{PV}

\newcommand{\p}{{\partial}}

\newcommand{\ph}{\phantom}
\newcommand{\nn}{\nonumber}

\newcommand{\R}{{\mathbb{R}}}

\newcommand{\N}{{\mathbb{N}}}

\renewcommand{\H}{{\mathcal{H}}}
\newcommand{\T}{{\mathbb{T}}}

\newcommand{\J}{{\mathbf{J}}}

\newcommand{\Sc}{{\mathcal{S}}}

\newcommand{\ep}{\epsilon}

\newcommand{\ul}{\underline}
\newcommand{\ux}{\underline{x}}

\renewcommand{\d}{\delta}

\newcommand{\ua}{\underline{a}}
\newcommand{\uy}{\underline{y}}

\newcommand{\tl}{\tilde}

\def\XXint#1#2#3{{\setbox0=\hbox{$#1{#2#3}{\int}$ }
\vcenter{\hbox{$#2#3$ }}\kern-.6\wd0}}

\title{Justification of the Point Vortex Approximation for Modified Surface Quasi-Geostrophic Equations}
\author{Matthew Rosenzweig}
\affil{University of Texas at Austin}
\affil{\textit{rosenzweig.matthew@math.utexas.edu}}

\begin{document}
\maketitle

\begin{abstract}
In this paper, we give a rigorous justification of the point vortex approximation to the family of modified surface quasi-geostrophic (mSQG) equations globally in time in both the inviscid and vanishing dissipative cases. This result completes the justification for the remaining range of the mSQG family unaddressed by Geldhauser and Romito \cite{GRmSQGPV2018} in the case of identically signed vortices.
\end{abstract}

\section{Introduction}
\label{sec:intro}
\subsection{Generalized Surface Quasi-Geostrophic Equation}
The \emph{surface quasi-geostrophic (SQG) equation} is the 2D active scalar equation
\begin{equation}
\label{eq:SQG}
\begin{cases}
\p_t\theta + u\cdot\nabla\theta + \kappa(-\Delta)^{\gamma/2}(\theta)= 0,  & (t,x) \in I\times \R^2 \\
u = (u_1,u_2) = (-R_2\theta,R_1\theta) & \\
\theta(0,x) = \theta^0(x), & x\in\R^2
\end{cases}.
\end{equation}
Here, $\kappa\geq 0$, $\gamma \geq 0$, $R_1$ and $R_2$ are the usual Riesz transforms (i.e. the Fourier multipliers $\frac{-i\xi_1}{|\xi|}$ and $\frac{-i\xi_2}{|\xi|}$), and $(-\Delta)^{\gamma/2}$ is the fractional Laplacian of order $\gamma$ (i.e. the Fourier multiplier $|\xi|^\gamma$). When $\kappa=0$, we refer to the equation as \emph{inviscid SQG}, and when $\kappa>0$, we refer to the equation as \emph{dissipative SQG}. Due to the relative balance between the dissipation term $\kappa (-\Delta)^{\gamma/2}(\theta)$ and the nonlinear term $u\cdot\nabla\theta$ under the scaling transformation $(t,x) \mapsto (\lambda t,\lambda x)$ for $\lambda >0$, we follow existing terminology in classifying the dissipative SQG as subcritical ($\gamma>1$), critical ($\gamma=1$), or supercritical ($\gamma\in [0,1)$).

The physical motivation for the inviscid SQG equation comes from the 3D quasi-geostrophic (QG) equation, which is an approximate description for the motion of a rotating stratified fluid with small Rosby number and small Froude number in which potential vorticity is conserved \cite{Pedlosky13, Lapeyre2017}. The SQG equation is obtained from the QG equation by assuming the potential vorticity is identically zero. In this setting the scalar $\theta$ represents temperature or surface buoyancy. The critical dissipative SQG equation is also physically relevant as the dissipation term $\kappa (-\Delta)^{1/2}(\theta)$ represents so-called Eckmann pumping \cite{Con02}.

The SQG equation was introduced to the mathematical community by Constantin, Majda, and Tabak \cite{cmt94}, who noted structural similarity to the 3D incompressible Euler in the form of vortex stretching. Global regularity for classical solutions to the 3D incompressible Navier-Stokes and Euler equations being a well-known open problem, the SQG equation has served as a toy model for studying singularity formation in fluid equations. Many results have been shown for the equation; we briefly mention a few which are germane to this article. Local existence and uniqueness of classical solutions to the inviscid and dissipative SQG equation in H\"{o}lder and Sobolev spaces is known \cite{cmt94, CCCGW2012gSQG}. Global existence (but crucially not uniqueness\footnote{See \cite{bsv16} for nonuniqueness of weak solutions to inviscid SQG in a certain class.}) of weak solutions to inviscid SQG is known in the spaces $L^p(\R^2)$, for $p \in (4/3,\infty)$ \cite{Res95, Marchand2008}. For the subcritical and critical dissipative SQG, classical solutions are known to be globally smooth \cite{CW1999QG, knvGWP07}, and in fact, weak solutions are smooth for positive times \cite{cv10, cvNMP2012, ctvFSQG2015}. For supercritical dissipative SQG, local existence and small data global existence of classical solutions is known \cite{Cordoba2004}, as well as global existence of weak solutions \cite{Marchand2008}.

The inviscid SQG equation has formal structure similar to the \emph{2D incompressible Euler equation} in vorticity form
\begin{equation}
\label{eq:E}
\begin{cases}
\p_t\omega + u\cdot \nabla\omega = 0, & (t,x)\in I\times\R^2 \\
u = -\nabla^{\perp}(\frac{1}{2\pi}\ln|\cdot| \ast \omega) \\
\omega(0,x) = \omega(x), & x\in\R^2
\end{cases},
\end{equation}
where $\nabla^\perp = (-\p_{x_2}, \p_{x_1})$. Note that the velocity field $u$ in \eqref{eq:E} is now obtained by convolution with a logarithmic potential rather than the more singular Riesz potential $\frac{1}{2\pi}|\cdot|^{-1}$ as in \eqref{eq:SQG}. By formally viewing \eqref{eq:E} as the equation of motion induced by the Hamiltonian 
\begin{equation}
\mathcal{H}_0(\omega) \coloneqq -\frac{1}{4\pi}\int_{\R^2\times\R^2}\ln|x-y|\omega(x)\omega(y)dxdy
\end{equation}
with respect to the Poisson bracket
\begin{equation}
\label{eq:PB}
\pb{F}{G}(\omega) = \int_{\R^2}\omega(x)\pb{\frac{\d F}{\d\omega}}{\frac{\d G}{\d\omega}}_{\R^2}(x)dx,
\end{equation}
where $F$ and $G$ are suitably differentiable real-valued functionals on an infinite-dimensional phase space, $\frac{\d F}{\d\omega}$ and $\frac{\d G}{\d\omega}$ are the variational derivatives of $F$ and $G$, respectively, and $\pb{}{}_{\R^2}$ is the standard Poisson structure for $\R^2$, then we obtain the inviscid SQG equation by instead considering the Hamiltonian
\begin{equation}
\H_{1}(\theta) \coloneqq \frac{1}{4\pi}\int_{\R^2\times\R^2}\frac{\theta(x)\theta(y)}{|x-y|}dxdy.
\end{equation}

The preceding observations lead us to consider the Hamiltonian equations of motion defined by
\begin{equation}
\H_\alpha(\theta) \coloneqq \frac{1}{2}\int_{\R^2\times\R^2}G_{\alpha}(x-y)\theta(x)\theta(y)dxdy, \qquad G_\alpha(x-y) \coloneqq \frac{1}{2^{(2-\alpha)}\pi}\frac{\Gamma(\frac{\alpha}{2})}{\Gamma(\frac{2-\alpha}{2})}|x-y|^{-\alpha}, \qquad \alpha \in (0,2)
\end{equation}
with respect to the Poisson structure given by \eqref{eq:PB}. Doing so, and allowing for dissipation which destroys the Hamiltonian structure, we obtain a family of 2D active scalar equations containing Euler and SQG:
\begin{equation}
\label{eq:gSQG}
\begin{cases}
\p_t\theta + u\cdot \nabla\theta + \kappa(-\Delta)^{\gamma/2}(\theta)= 0, & (t,x) \in I\times\R^2 \\
u = \nabla^{\perp}(G_\alpha\ast \theta) & \\
\theta(0,x) = \theta^0(x), & x\in\R^2
\end{cases}.
\end{equation}
This family of equations, first introduced by Pierrehumbert et al \cite{PHS1994spec} to study 2D turbulence, is sometimes referred to in the literature as \emph{generalized inviscid/dissipative SQG (gSQG)}, while the sub-family $\alpha \in (0,1)$ is typically called \emph{inviscid/dissipative modified SQG (mSQG)}. That the $\alpha\rightarrow 0^+$ limit should be the Euler equation can be seen from elementary convergence result
\begin{equation}
\lim_{\alpha\rightarrow 0^+} 2^{-(2-\alpha)}\frac{\Gamma(\frac{\alpha}{2})}{\Gamma(\frac{2-\alpha}{2})}\int_{\R^2}|x|^{-\alpha}f(x)dx = -\frac{1}{2\pi}\int_{\R^2}\ln|x|f(x)dx
\end{equation}
for any Schwartz function $f$ with mean value zero. As before, we can classify the dissipative gSQG equation as subcritical, critical, or supercritical based on the parameters $\alpha$ and $\gamma$.

\subsection{Point Vortex Model}
We now discuss a finite-dimensional model intimately related to the gSQG family of equations. For a positive integer $N$, the \emph{gSQG point vortex model} is the system of $N$ ordinary differential equations (ODEs)
\begin{equation}
\label{eq:PVM}
\begin{cases}
\dot{x}_{i}(t) = \sum_{1\leq i\neq j\leq N} a_j \nabla^\perp G_\alpha(x_i(t)-x_j(t)) \\
x_i(0) = x_i^0
\end{cases},
\qquad i\in\{1,\ldots,N\},
\end{equation}
where $a_1,\ldots,a_N\in\R\setminus\{0\}$ are the intensities and pairwise distinct $x_1^0,\ldots,x_N^0\in\R^2$ are the initial positions. To connect the system \eqref{eq:PVM} to the inviscid gSQG equation \eqref{eq:gSQG}, we may argue informally as follows. First, note that by the anti-symmetry of the Biot-Savart kernel $K_\alpha \coloneqq \nabla^{\perp}G_\alpha$, we obtain the following symmetrized weak formulation of gSQG:
\begin{equation}
\label{eq:wvf}
\begin{split}
\int_{\R^2}\theta(t,x)f(t,x)dx &= \int_{\R^2}\theta^0(x)f(0,x)dx \\
&\ph{=} + \frac{1}{2}\int_0^t \PV\int_{\R^2\times\R^2}\theta(t,x)\theta(t,y) K_{\alpha}(x-y)\cdot [\nabla f(s,x) - \nabla f(s,y)]dxdyds \\
&\ph{=} -\int_0^t\int_{\R^2}\theta(s,x)\p_t f(s,x)dxds,
\end{split}
\end{equation}
for any spacetime test function $f$. Due to the presence of the principal value in defining the singular integral, it is evident that the weak formulation \eqref{eq:wvf} makes sense for measure-valued solutions of the form
\begin{equation}
\label{eq:ed}
\theta(t,x) = \sum_{i=1}^{N}a_i \delta_{x_i(t)}(x).
\end{equation}
A quick computation then shows that $\theta$ satisfies equation \eqref{eq:wvf} if and only if $x_1,\ldots,x_N$ satisfy equation \eqref{eq:PVM}. As noted by Marchioro and Pulvirenti in their classic reference \cite{MP2012book}, the physical intuition is that a single point vortex moves only under the velocity fields generated by the other point vortices but not under the velocity field generated by itself (i.e. there is no self-interaction). In particular, a system of one point vortex does not move at all.

The classical point vortex model corresponding to the 2D Euler equation (i.e. $\alpha=0$ in equation \eqref{eq:PVM}) goes back to Helmholtz \cite{Helmholtz1858} and Kirchoff \cite{Kirchoff1876} and has been studied both as a toy model for 2D flow and as a dynamical system in its own right (see \cite{MP1984ln, MP2012book} for survey of results). The point vortex model corresponding to gSQG equations, in particular SQG, is a more recent topic of study. Compared to Euler point vortices, gSQG point vortices rotate about each other more rapidly. Moreover, due to the faster decay of the Biot-Savart kernel $K_\alpha$ away from the origin as $\alpha$ increases, there is a greater tendency in the gSQG model to form localised vortex clusters (with tendency increasing as $\alpha$ increases) \cite{HPGS1995}. gSQG point vortices also demonstrate different behavior than Euler point vortices in terms of vortex collapse. While a system of three Euler point vortices can collapse in finite time only in a self-similar fashion, the analogous gSQG system can collapse in finite time in either self-similar or non self-similar fashion \cite{BB2018}. Lastly, we mention that while finite-time collapse (the analogue of blow-up of solutions for PDEs) is possible for gSQG point vortices for all $\alpha\geq 0$, global existence neverthless holds for Euler \cite{MP2012book} and mSQG point vortices \cite{GRmSQGPV2018} for "generic" initial data. Just as global existence of classical solutions for the gSQG equation, for $\alpha \geq 1$, is an open problem, global existence of solutions to the gSQG point vortex model, for $\alpha\geq 1$, with "generic" initial data is currently unknown.

We have given a heuristic derivation of the point vortex model from the corresponding gSQG equation. To obtain a rigorous derivation, we might first try to use well-posedness theory for the gSQG equation to model the evolution of singular initial data of the form
\begin{equation}
\label{eq:ID_form}
\sum_{i=1}^N a_i\delta_{x_i^0},
\end{equation}
where $N\in\N$, $a_1,\ldots,a_N\in\R\setminus\{0\}$, and $x_1^0,\ldots,x_N^0\in\R^2$ are pairwise distinct. Unfortunately, to our knowledge there is no deterministic well-posedness theory for initial data of the form \eqref{eq:ID_form}. For 2D Euler, Delort proved existence (but not uniqueness) of weak solutions in $H_{loc}^{-1}(\R^2)$ for initial datum $\theta^0$ which is a positive Radon measure with velocity field $u^0 \in L_{loc}^2$ \cite{Delort1991} (see also \cite{Majda1993, Schochet1995} for alternative proofs and related results). However, Delort's result is not applicable to the study of initial data \eqref{eq:ID_form} since the initial velocity field is proportional to the vector
\begin{equation}
\sum_{j=1}^N \frac{(x-x_j^0)^{\perp}}{|x-x_j^0|^2},
\end{equation}
which is not square-integrable in a neighborhood of $x_j^0$, for $j\in\{1,\ldots,N\}$. For the gSQG equation with $\alpha>0$, it is not clear that the weak formulation \eqref{eq:wvf} even makes sense for general Radon measures due to the possibility of concentration along the diagonal. Moreover, to our knowledge there is no analogue of Delort's existence result for such equations. While there are probabilistic well-posedness results for Euler and mSQG (but not for SQG and more singular gSQG equations) at very low regularities \cite{ac1990global, NPST2018mSQG, Flandoli2017wv, FS2018mSQG}, these results are also not applicable for at least two reasons. First, they do not contain a uniqueness assertion for initial data in the support of the measure for the underlying randomization scheme. Hence, we do not know if the solution such a scheme constructs has the desired form \eqref{eq:ed}. Second, and more importantly, we do not know if our initial data lie in the support of the measures these results use. Since starting directly from initial data \eqref{eq:ID_form} is seemingly hopeless, we should instead try to rigorously justify the gSQG point vortex model from the gSQG equation by showing that sequences $\theta_\epsilon^0$ of regularized initial data which weak-* converge to \eqref{eq:ID_form} have corresponding solutions $\theta_\ep$ which weak-* converge to \eqref{eq:ed}. Following the terminology of \cite{MP1993vl}, we refer to such a result as \emph{localization of the flow}.

Considering first the inviscid case, suppose that we have a sequence $\{\theta_\epsilon^0\}_{\epsilon>0}$ of sufficiently regular initial data of the form $\theta_\epsilon^0 = \sum_{i=1}^{N} \theta_{i,\epsilon}^0$, where
\begin{itemize}
\item
\begin{equation}
\supp(\theta_{i,\epsilon}^0) \subset B(x_i^0, C\epsilon), \qquad i\in\{1,\ldots,N\},
\end{equation}
for some $C>0$ independent of $\epsilon$, and
\begin{equation}
\min_{1\leq i< j\leq N} \dist(\supp(\theta_{i,\epsilon}) , \supp(\theta_{j,\epsilon})) \gtrsim 1;
\end{equation}
\item
\begin{equation}
\int_{\R^2} \theta_{i,\epsilon}^0 = a_{i} \in\R\setminus\{0\}, \qquad \forall \epsilon>0, \enspace i\in\{1,\ldots,N\};
\end{equation}
\item
\begin{equation}
0\leq \theta_{i,\epsilon}^0 \lesssim \epsilon^{-2}, \qquad i\in\{1,\ldots,N\}.
\end{equation}
\end{itemize}

\begin{remark}
We easily obtain an example of a sequence $\{\theta_\epsilon^0\}$ of initial data satisfying the preceding conditions by considering a bump function $\chi$ with support in the unit ball $B(0,1)$ and defining
\begin{equation}
\theta_{i,\epsilon}^0(x) \coloneqq \frac{4a_i}{(d_0(\ux_N^0)\epsilon)^2\|\chi\|_{L^1(\R^2)}}\chi\paren*{\frac{2(x-x_i^0)}{d_0(\ux_N^0)\epsilon}},
\end{equation}
where $d_0(\ux_N^0) \coloneqq \min_{1\leq i< j\leq N} |x_i^0-x_j^0|$, and setting $\theta_\epsilon^0 \coloneqq \sum_{i=1}^N \theta_{i,\epsilon}^0$.
\end{remark}

Now let $\theta_{\epsilon}$ denote the solution of the gSQG equation with initial datum $\theta_{\epsilon}^0$, and let $\ux_N=(x_1,\ldots,x_N)$ denote the solution gSQG point vortex model with intensities $\ua_N=(a_1,\ldots,a_N)$, and initial positions $\ux_N^0=(x_1^0,\ldots,x_N^0)$. We say that the flow \emph{localizes} on the interval $[0,T]$ if for every $t\in [0,T]$,
\begin{equation}
\label{eq:wc}
\theta_{\epsilon}(t) \xrightharpoonup[\epsilon\rightarrow 0^+]{*} \sum_{i=1}^N a_i\delta_{x_i(t)}
\end{equation}
in the space $\mathcal{M}(\R^2)$ of finite Borel measures endowed with the weak-* topology.

Considering now the dissipative case, if the term $\kappa(-\Delta)^{\gamma/2}\theta_\ep$ remains of size $O(1)$ compared to the diameter of the support of $\theta_{i,\epsilon}^0$, which is of size $O(\epsilon)$, then it is not clear that we should expect localization in the limit as $\epsilon\rightarrow 0^{+}$. Accordingly, we should instead consider the gSQG equation with \emph{vanishing dissipation}:
\begin{equation}
\label{eq:gSQG_vd}
\p_t\theta_{\epsilon} + u_{\epsilon}\cdot\nabla \theta_{\epsilon} + \kappa(\epsilon)(-\Delta)^{\gamma/2}\theta_{\epsilon}=0,
\end{equation}
where $\kappa: [0,\infty)\rightarrow [0,\infty)$ is some sufficiently decreasing function such that $\kappa(0)=0$.

We now comment on the existing results regarding localization of gSQG flows. To our knowledge, rigorous results go back to the 1980s beginning with \cite{MP1983si}. In this work, Marchioro and Pulvirenti showed localization of the 2D Euler equation for short times for an arbitrary number of vortices in both $\R^2$ and bounded domains. Turkington \cite{Turkington1987} established the first global result by proving localization of the 2D Euler flow in the case of single vortex initial data $\theta_\epsilon^0$ with a definite sign in a bounded, simply connected domain. His result is also valid in unbounded domains. Turkington's result was extended to the case of two vortices of different sign by Marchioro and Pagani in \cite{MPag1986} and to the case of $N$ vortices of the same sign by Marchioro in \cite{Marchioro1988e}. Building upon the ideas of the latter work, Marchioro also proved localization of the 2D Navier-Stokes flow in the vanishing viscosity limit for the case of $N$ vortices of the same sign \cite{Marchioro90vNS}. Global in time localization for 2D Euler in the general case of $N$ vortices of arbitrary sign was finally solved by Marchioro and Pulvirenti \cite{MP1993vl}. Following the strategy of \cite{MP1993vl} with some additional Harmonic Analaysis to handle the singularity of the Biot-Savart kernel, Geldhauser and Romito \cite{GRmSQGPV2018} extended that work's result to inviscid mSQG flows with parameter $\alpha \in [0,2-\sqrt{3})$. We also mention the work \cite{CGM2013} which considered localization of surface quasi-geostrophic flows.

\subsection{Main Result}
Having properly discussed localization of gSQG flows and reviewed the existing results concerning the problem, we are now prepared to state our main result.

\begin{thm}[mSQG localization]
\label{thm:main_mSQG}
Let $\alpha \in [0,1)$, $\gamma \in [0,2]$, $N\in\N$, and $T>0$. Let $a_1,\ldots,a_N\in\R\setminus\{0\}$ have identical sign, and let $x_1^0,\ldots,x_N^0$ be pairwise distinct elements of $\R^2$. Let $\kappa: (0,\infty)\rightarrow [0,\infty)$ be a function such that
\begin{equation}
\lim_{\ep\rightarrow 0^+} \frac{\kappa(\ep)^{\frac{2-\alpha}{2}}}{\ep^\alpha} = 0.
\end{equation}
Suppose that we have a sequence $\{\theta_{1,\ep}^0, \ldots,\theta_{N,\ep}^0\}_{0<\ep\leq 1}$ of initial data satisfying the following conditions:
\begin{enumerate}[(i)]
\item
For $i\in\{1,\ldots,N\}$, $\theta_{i,\ep}^0 \in C_c^\infty(\R^2)$, $\supp(\theta_{i,\ep}^0) \subset B(x_i^0, C_1\ep)$, for some constant $C_1>0$ independent of $\ep$, and
\begin{equation}
\min_{1\leq i < j\leq N}\dist(\supp(\theta_{i,\ep}^0), \supp(\theta_{j,\ep}^0)) \geq d,
\end{equation}
for some constant $d>0$ independent of $\ep$;
\item
$\int_{\R^2}\theta_{i,\ep}^0(x)dx = a_i$ for all $0<\ep\leq 1$ and $i\in\{1,\ldots,N\}$;
\item
there exists a constant $C_2>0$ such that $0\leq \theta_{i,\ep}^0\leq C_2\ep^{-2}$ for all $0<\ep\leq 1$ and $i\in\{1,\ldots,N\}$.
\end{enumerate}
Let $\theta_\ep$ denote the unique solution to the gSQG equation \eqref{eq:gSQG} with parameters $(\alpha,\gamma,\kappa(\ep))$ starting from the initial datum
\begin{equation}
\theta_\ep^0 \coloneqq \sum_{i=1}^N \theta_{i,\ep}^0,
\end{equation}
and suppose also that the interval $[0,T]$ is contained in the lifespan of $\theta_\ep$ for all $\ep>0$ sufficiently small. Let $\ux_N(t) = (x_1(t),\ldots,x_N(t))$ denote the solution of the gSQG point vortex model \eqref{eq:PVM} with parameter $\alpha$. Then for any bounded, continuous function $f$, we have that
\begin{equation}
\lim_{\ep\rightarrow 0^+} \sup_{0\leq t\leq T} \left|\int_{\R^2}\theta_{\ep}(t,x)f(x)dx - \sum_{i=1}^N a_i f(x_i(t))\right| = 0.
\end{equation}
\end{thm}

\cref{thm:main_mSQG} resolves the problem of justifying mSQG point vortices of identical sign from the inviscid mSQG equation in the range $\alpha \in [2-\sqrt{3},1)$ left open by Geldhauser and Romito \cite{GRmSQGPV2018}. Moreover, one can view this result as a converse to Duerinckx's derivation of the inviscid mSQG family as the mean field limit of the point vortex model \eqref{eq:PVM} with parameter $\alpha\in [0,1)$ \cite{Duerinckx2016} (see also \cite{Serfaty2018mean}).

We conclude this subsection with some remarks about the assumptions in the statements of \cref{thm:main_mSQG} and some further possible extensions to less regular classes of solutions and domains different than $\R^2$.

\begin{remark}
\label{rem:G}
Under our assumption that the intensities $a_1,\ldots,a_N$ have the same sign, the gSQG point vortex model with parameter $\alpha\geq 0$ has a unique global solution for arbitrary initial positions $x_1^0,\ldots,x_N^0$ as a consequence of conservation of the Hamiltonian. We elaborate on this point in \cref{ssec:PVM_GWP}. Therefore, \cref{thm:main_mSQG} is only conditional on the lifespans of the solution sequences $\{\theta_{\ep}\}_{0<\ep\leq 1}$. For inviscid gSQG and supercritical dissipative gSQG, it is currently unknown whether classical solutions are global. If we just consider the case $N=1$ and fix a  bump function $\theta^0$ centered at the origin, then defining the sequence of initial data $\theta_\ep^0 \coloneqq \ep^{-2}\theta^0(\ep^{-1}\cdot)$, some scaling analysis and the uniqueness of classical solutions imply that the solution $\theta_\ep$ to inviscid gSQG with initial datum $\theta_\ep^0$ is given by
\begin{equation}
\theta_\ep(t,x) = \ep^{-2}\theta(\ep^{-2-\alpha}t,\ep^{-1}x),
\end{equation}
where $\theta$ is the solution to the inviscid gSQG equation with initial datum $\theta^0$. If $\theta$ has maximal lifespan $[0,T)$, then $\theta_\ep$ has maximal lifespan $[0,\ep^{2+\alpha}T)$. So in order for the sequence $\{\theta_\ep^0\}_{0<\ep\leq 1}$ to satisfy the conditions of \cref{thm:main_mSQG}, we need the qualitative assumption that $\theta$ is a global solution. For subcritical and critical dissipative mSQG, we know from the work \cite{MX2012} that classical solutions are global; hence, the derivation from such equations is unconditional.
\end{remark}

\begin{remark}
\label{rem:reg}
The regularity assumptions for the initial data $\{\theta_{i,\ep}^0\}_{1\leq i\leq N, 0<\ep\leq 1}$ can be lowered substantially. In the work \cite{MP1993vl}, the authors work with the unique global $L^1 \cap L^\infty$ weak solutions due to Yudovich \cite{Yudovich1963}, and in the work \cite{GRmSQGPV2018}, the authors construct global (but not necessarily unique) $L^1\cap L^\infty$ weak solutions with which they work. Applying our method of proof in the case $\alpha \in (0,1)$ with the weak solutions of \cite{GRmSQGPV2018} should not present a difficulty.
\end{remark}

\begin{remark}
\label{rem:dom}
The proof of \cref{thm:main_mSQG} carries over to the periodic case where $\R^2$ is replaced by the 2-torus $\T^2$. Although in the periodic case we no longer have an explicit formula for the Green's function $G_\alpha$ and the Biot-Savart kernel $K_\alpha$, we still have the same estimates as in the free space case, which is all that we need. Extending our result to the case of more general smooth, bounded domains is more delicate: there are several inequivalent definitions of the fractional Laplacian for a bounded domain, leading to different Green's functions $G_\alpha$ and Biot-Savart kernels $K_\alpha$.
\end{remark}

\subsection{Outline of the Proof}
Our proof of \cref{thm:main_mSQG} is inspired by Marchioro's argument for justifying the Euler point vortex model from the 2D Navier-Stokes equation in the vanishing viscosity limit \cite{Marchioro90vNS}. Namely, we consider an approximate moment of inertia of the form
\begin{equation}
I_{R,\ep}(t) \coloneqq \sum_{i=1}^N \int_{\R^2}|x-x_i(t)|^2\phi_R(x-x_i(t))\theta_{\epsilon}(t,x)dx,
\end{equation}
where $\phi_R$ is a suitably chosen bump function localized to a ball of radius $\sim R$,\footnote{If the cutoff function $\phi_R$ were absent, then the resulting quantity would be conserved; however, $I_{R,\ep}(0)$ would then be of size $\sim 1$ in general. The cutoff is present to ensure that $I_{R,\ep}(0)$ is of size $\sim \ep^2$. Note that if $N=1$ (i.e. there is only a single vortex patch), then we can use the genuine moment of inertia and prove localization of the gSQG flow, for every $\alpha\in [0,2]$. However, since a single point vortex is stationary, such a result is not so physically interesting.} and define an associated maximal function
\begin{equation}
\bar{I}_{R,\ep}(t) \coloneqq \sup_{0\leq s\leq t} I_{R,\ep}(s).
\end{equation}
To simplify the ensuing discussion, suppose that $\kappa(\ep) = o(\ep^2)$. Our goal then is to show that there exist constants $\ep_0 >0$ and $D>1$ depending on the ambient parameters, including $R$ but crucially not $\ep$, such that
\begin{equation}
\label{eq:intro_Iest}
\bar{I}_{R,\ep}(T) \leq D\ep^2
\end{equation}
for all $0< \ep\leq \ep_0$. The utility of such an estimate stems from the elementary observation that for any radius $r\in (0,R)$,
\begin{equation}
\sum_{i=1}^N \int_{|x-x_i(t)|\geq r} \phi_R(x-x_i(t))\theta_\ep(t,x)dx \leq \frac{\bar{I}_{R,\ep}(t)}{r^2} \leq \frac{D\ep^2}{r^2}, \qquad t\in [0,T], \enspace \ep \in (0,\ep_0],
\end{equation}
together with the fact that $\bar{I}_{R,\ep}(t)$ controls the quantity
\begin{equation}
\max_{i=1,\ldots,N} \left|a_i - \int_{|x-x_i(t)|\leq R}\theta_\ep(t,x)dx\right|.
\end{equation}
By appropriately choosing the radius $r$ to vanish as $\ep\rightarrow 0^+$, we can then show that for a given bounded, continuous function $f$ and $t\in [0,T]$,
\begin{equation}
\lim_{\ep\rightarrow 0^+} \left|\int_{\R^2}f(x)[\theta_\ep(t,x) - \sum_{i=1}^Na_i\delta_{x_i(t)}]dx\right| = 0,
\end{equation}
which is precisely the desired weak-* convergence result.

To prove such an estimate \eqref{eq:intro_Iest}, we use the energy method. The chief difficulty in our setting compared to that of 2D Euler or Navier-Stokes is the singularity of the Biot-Savart kernel $K_\alpha$. For $\alpha=0$, the function
\begin{equation}
(x-y)\cdot K_\alpha(x-y)
\end{equation}
is bounded almost everywhere and smooth outside the diagonal $x=y$. For $\alpha>0$, this is no longer the case and instead we only have the estimate
\begin{equation}
|(x-y)\cdot K_\alpha(x,y)| \lesssim_\alpha \frac{1}{|x-y|^\alpha},
\end{equation}
which necessitates bounding expressions of the form
\begin{equation}
\label{eq:intro_rp_ex}
\int_{|x-x_i(t)|\sim R}\int_{|y-x_i(t)| \sim R}\frac{\theta_\ep(t,x)\theta_\ep(t,y)}{|x-y|^\alpha}dxdy.
\end{equation}
Since we are unable to rule out that the scalar may spread to distance $\sim R$ away from $x_i(t)$, we resort to bounding the modulus of \eqref{eq:intro_rp_ex} directly. To do so, we use the Fubini-Tonelli theorem to reduce to bounding the $L^\infty$ norm of a Riesz potential operator. As the $L^p$ norms are conserved in the case of inviscid gSQG and are nonincreasing in the case of dissipative gSQG, we want to bound the quantity
\begin{equation}
\sup_{x\in\R^2} \left|\int_{S}\frac{\theta_\ep(t,y)}{|x-y|^\alpha}dy\right|,
\end{equation}
for a given measurable set $S$, in terms of $\|\theta_{\ep}\|_{L^p(S)}$ for suitable $p$. For this task, we rely on the potential theory result \cref{prop:Linf_RP}, which gives a bound of the form
\begin{equation}
\sup_{x\in\R^2} \left|\int_{S}\frac{\theta_\ep(t,y)}{|x-y|^\alpha}dy\right| \lesssim_\alpha \|\theta_\ep^0\|_{L^1(S)}^{\frac{2-\alpha-\frac{2}{p}}{2-\frac{2}{p}}} \|\theta_\ep^0\|_{L^p(S)}^{\frac{\alpha}{2-\frac{2}{p}}}.
\end{equation}
This estimate is precisely the source of the restriction to the range $\alpha \in [0,1)$ and the assumption on the dissipation coefficient $\kappa(\ep)$. Indeed, while $\|\theta_\ep^0\|_{L^1(\R^2)} = \sum_{i=1}^N a_i$ by assumption, $\|\theta_\ep^0\|_{L^p(\R^2)} \sim \ep^{-2(1-\frac{1}{p})}$ which blows up as $\ep\rightarrow 0^+$. Ultimately, we prove an inequality for $\bar{I}_{R,\ep}$ which in caricature reads
\begin{equation}
\label{eq:intro_I_G}
\bar{I}_{R,\ep}(t) \leq I_{R,\ep}(0) + C(\alpha,\gamma)\kappa(\ep)N R^{2-\gamma}t + \frac{C(\alpha,\gamma)}{d_T(\ux_N^0)^{2+\alpha}}\int_0^t \bar{I}_{R,\ep}(s)ds + \frac{C(\alpha,\gamma)}{\ep^\alpha R^4}\int_0^t s^{\frac{4-\alpha}{2}} \bar{I}_{R,\ep}(s)^{\frac{4-\alpha}{2}}ds.
\end{equation}

To use inequality \eqref{eq:intro_I_G} to obtain the desired estimate \eqref{eq:intro_Iest} via the Gronwall-Bellman inequality (see \cref{prop:GB}), we need to deal with the fact that there is a factor of $\ep^{-\alpha}$ appearing in the RHS of \eqref{eq:intro_I_G} in addition to a power of $\bar{I}_{R,\ep}(s)$ with exponent bigger than $1$. Therefore, we proceed by a bootstrap argument. First, we restrict to a subinterval $[0,T_*]\subset [0,T]$ on which $\bar{I}_{R,\ep}(t) \leq D\ep^2$, for a given constant $D>1$, the precise choice of which we later optimize. On the subinterval $[0,T_*]$, we then have the bound
\begin{equation}
\frac{\bar{I}_{R,\ep}(s)^{1+\frac{2-\alpha}{2}}}{\ep^\alpha} \leq D^{\frac{2-\alpha}{2}}\ep^{2(1-\alpha)}\bar{I}_{R,\ep}(s),
\end{equation}
which we apply to the RHS of \eqref{eq:intro_I_G}. For $\alpha \in (0,1)$, we have some extra room in the exponent of $\ep$ so that upon application of Gronwall-Bellman, we can obtain a lower bound for $T_*$ which implies, upon choosing $D=D(R,\alpha,\gamma,C_1,C_2,N,\kappa,T)>0$ sufficiently large and $\ep=\ep(D)>0$ sufficiently small, that $T_*\geq T$. For $\alpha=1$, we have no such extra room, and therefore we can only obtain a positive lower bound for $T_*$ which depends on the data $(R,\alpha,\gamma,C_1,C_2,N,\kappa,T)$. The dependence of the lower bound bound on $R$ and the need to let $R\rightarrow 0^+$ prevents us from proving localization of the SQG flow.\footnote{The lower bound tends to zero as $R\rightarrow 0^+$.}

We close this subsection with a comparison of our proof to that of Geldhauser and Romito in \cite{GRmSQGPV2018}. They follow the strategy of \cite{MP1993vl}. Namely, the authors work perturbatively by singling out a single patch $\theta_{i_0,\ep}$ and treating the velocity field generated by the remaining $N-1$ patches as a Lipschitz external field, which is valid provided that the patches remain at distance $\sim 1$ from one another on the given time interval $[0,T]$. In order to prove weak-* convergence, they first prove a concentration estimate showing the solution has ``small" density outside a ball of radius $\sim \rho(\ep)$ centered at the perturbed systems center of vorticity $c_\ep(t)$, where $\rho$ is some function of $\ep$ vanishing as $\ep\rightarrow 0^+$. They then upgrade this estimate to compact support in a ball of radius $\tl{\rho}(\ep)$, where $\tl{\rho}$ is some function vanishing as $\ep\rightarrow 0^+$, by proving an $L^\infty$ estimate for the component of the velocity field normal to the boundary of a small disk around $c_\ep(t)$. This last step uses the Hardy-Littlewood-Sobolev lemma (see \cref{prop:HLS}), which introduces an inefficiency in the argument leading to the restriction on $\alpha$. Note that this last step is necessary in order to bootstrap the initial perturbative assumption to the entire interval $[0,T]$. Our approach bypasses the need for an $L^\infty$ estimate on the velocity field at the cost of only being able to treat the case of intensities of identical sign. Replacing $\theta_\ep$ by $|\theta_\ep|$ in the definition of $I_{R,\ep}$ would destroy the good symmetry properties of the numerous double integrals we have to estimate. 

\subsection{Organization of the Paper}
Having outlined the proof of our main results in the preceding subsection, we now comment on the organization of the manuscript.

In \cref{sec:pre}, we introduce basic notation used throughout this paper. We also recall some well-known facts from Harmonic Analysis concerning singular integral operators, such as Riesz potentials and fractional Laplacians, used in our moment of inertia estimate. Next, we state the precise form of the Gronwall-Bellman inequality used in this article. Lastly, we recall some properties of solutions to the inviscid and dissipative gSQG equations concerning $L^p$ norms and maximum/minimum principles.

In \cref{sec:PVM}, we discuss some basic properties of the gSQG point vortex model which are largely generalizations of classical results for the Euler point vortex model. We begin by reviewing the Hamiltonian structure of system \eqref{eq:PVM} in addition to discussing some conserved quantities for the equation. Since there does not appear to be a self-contained proof in the literature, we show local well-posedness of system \eqref{eq:PVM}. We close the section with some remarks on the global well-posedness for various values of $\alpha$.

In \cref{sec:main_est}, we prove \cref{prop:Ibar}, which is the essential moment of inertia inequality for proving \cref{thm:main_mSQG}. Since the proof is technical and depends on various sub-estimates and case analysis, we have divided it into subsections and sub-subsections corresponding to the various decompositions introduced throughout the proof.

In \cref{sec:boot}, we use \cref{prop:Ibar} and a bootstrap argument to prove \cref{prop:b_MI} and \cref{cor:ai_diff}, which gives a concentration estimate. Lastly, in \cref{sec:pf_main}, we use the final moment of inertia estimate given by \cref{prop:b_MI} and the concentration estimate of \cref{cor:ai_diff} to prove \cref{thm:main_mSQG}.

\subsection{Acknowledgements}
The author would like to thank his advisor, Nata\v{s}a Pavlovi\'c, and colleague Matthew Novack for several enlightening discussions on surface quasi-geostrophic equations. The author gratefully acknowledges financial support through NSF Grant DMS-1516228 and a Provost Graduate Excellence Fellowship from the University of Texas at Austin.

\section{Preliminaries}
\label{sec:pre}
\subsection{Notation}
In this subsection, we introduce some basic notation used throughout the paper.

Given nonnegative quantities $A$ and $B$, we will write $A\lesssim B$ if there exists a constant $C>0$, independent of $A$ and $B$, such that $A\leq CB$. If $A \lesssim B$ and $B\lesssim A$, we will write $A\sim B$. To emphasize the dependence of the constant $C$ on some parameter $p$, we will sometimes write $A\lesssim_p B$ or $A\sim_p B$.

Given points $x_1,\ldots,x_N$ in some set $X$, we will write $\ux_N$ to denote the $N$-tuple $(x_1,\ldots,x_N)$. We define the generalized diagonal $\Delta_N$ of the Cartesian product $X^N$ to be the set
\begin{equation}
\Delta_N \coloneqq \{(x_1,\ldots,x_N) \in X^N : x_i=x_j \text{ for some $i\neq j$}\}.
\end{equation}
Given $x\in\R^n$ and $r>0$, we denote the ball centered at $x$ of radius $r$ by $B(x,r)$. Given a function $f$, we denote the support of $f$ by $\supp(f)$. Given $h\in \R^n$, we define the translation operator $\tau_h$ by $(\tau_h f)(x) \coloneqq f(x-h)$. On $\R^2$, we define the perpendicular gradient $\nabla^\perp = (-\p_{x_2},\p_{x_1})$, where $\p_{x_j}$ is the partial derivative in the $x_j$-direction.

We denote the Banach space of real-valued continuous, bounded functions on $\R^n$ by $C(\R^n)$ equipped with the uniform norm $\|\cdot\|_{\infty}$. More generally, we denote the Banach space of $k$-times continuously differentiable functions with bounded derivatives up to order $k$ by $C^k(\R^n)$ equipped with the natural norm, and we define $C^\infty \coloneqq \bigcap_{k=1}^\infty C^k$. We denote the subspace by of smooth functions with compact support by $C_c^\infty(\R^n)$. We define the space of locally continuous functions by $C_{loc}(\R^n)$ and similarly for locally $C^k$ and locally smooth functions. We denote the Schwartz space of functions by $\Sc(\R^n)$. For $p\in [1,\infty]$ and $\Omega\subset\R^n$, we define $L^p(\Omega)$ to be the Banach space equipped with the norm
\begin{equation}
\|f\|_{L^p(\Omega)} \coloneqq \paren*{\int_\Omega |f(x)|^p dx}^{1/p}
\end{equation}
with the obvious modification if $p=\infty$. When $\Omega=\R^n$, we will sometimes just write $\|f\|_{L^p}$. Lastly, we use the notation $\|\cdot\|_{\ell^p}$ in the conventional manner.

Our convention for the Fourier transform and inverse transform are respectively
\begin{equation}
\mathcal{F}(f)(\xi) = \hat{f}(\xi) = \int_{\R^n}f(x)e^{-i\xi\cdot x}dx
\end{equation}
and
\begin{equation}
\mathcal{F}^{-1}(f)(x) = f^{\vee}(x) = \frac{1}{(2\pi)^n}\int_{\R^n}f(\xi)e^{i x\cdot\xi}d\xi.
\end{equation}

\subsection{Harmonic Analysis}
In this subsection, we review some basic results from Harmonic Analysis concerning singular integrals. Standard references for the subject matter are \cite{Stein1970, grafakos2014c, grafakos2014m}.

For $s>-n$, we define the Fourier multiplier $(-\Delta)^{s/2}$ by
\begin{equation}
(-\Delta)^{s/2}f(x) = (|\xi|^{s}\hat{f}(\xi)^\vee(x), \qquad x\in \R^n,
\end{equation}
for a Schwartz function $f\in \Sc(\R^n)$. Since, for $s\in (-n,0)$, the inverse Fourier transform of $|\xi|^s$ is the tempered distribution
\begin{equation}
\frac{2^s\Gamma(\frac{n+s}{2})}{\pi^{\frac{n}{2}}\Gamma(-\frac{s}{2})} |x|^{-s-n},
\end{equation}
it follows that
\begin{equation}
(-\Delta)^{s/2}f(x) = \frac{2^s \Gamma(\frac{n+s}{2})}{\pi^{\frac{n}{2}}\Gamma(-\frac{s}{2})}\int_{\R^n}\frac{f(y)}{|x-y|^{s+n}}dy, \qquad x\in\R^n.
\end{equation}
For $s\in (0,n)$, we define the \emph{Riesz potential operator} of order $s$ by $\mathcal{I}_s \coloneqq (-\Delta)^{-s/2}$ on $\Sc(\R^n)$. $\mathcal{I}_s$ extends to a well-defined operator on any $L^p$ space, the extension also denoted by $\mathcal{I}_s$ by an abuse notation, as a consequence of the \emph{Hardy-Littlewood-Sobolev lemma}.

\begin{prop}[Hardy-Littlewood-Sobolev]
\label{prop:HLS}
Let $s \in (0,n)$, and let $1<p<q<\infty$ satisfy the relation
\begin{equation}
\frac{1}{p}-\frac{1}{q} = \frac{s}{n}.
\end{equation}
Then for all $f\in\Sc(\R^n)$,
\begin{equation}
\|\mathcal{I}_s(f)\|_{L^q(\R^n)} \lesssim_{n,s,p} \|f\|_{L^p(\R^n)}
\end{equation}
and
\begin{equation}
\|\mathcal{I}_s(f)\|_{L^{\frac{n}{n-1},\infty}(\R^n)} \lesssim_{n,s} \|f\|_{L^1(\R^n)},
\end{equation}
where $L^{r,\infty}$ denotes the weak-$L^r$ space. Consequently, $\mathcal{I}_s$ has a unique extension to $L^p$, for all $1\leq p<\infty$.
\end{prop}

Although the Hardy-Littlewood-Sobolev lemma breaks down at the endpoint case $p=\infty$ (one has a $BMO$ substitute which is not useful for our purposes), the next lemma allows us to control the $L^\infty$ norm of $\mathcal{I}_s(f)$ in terms of the $L^1$ norm and $L^p$ norm, for some $p=p(s,d)$. One can view this result as a poor man's version of Hedberg's inequality \cite{Hedberg1972}, which gives a pointwise estimate for $\mathcal{I}_s(f)$.

\begin{prop}[$L^{\infty}$ bound for Riesz potential]
\label{prop:Linf_RP}
For any $s\in (0,n)$ and $p\in (\frac{n}{s},\infty]$,
\begin{equation}
\|\mathcal{I}_s(f)\|_{L^\infty(\R^n)} \lesssim_{s,n,p} \|f\|_{L^{1}(\R^n)}^{1-\frac{n-s}{n(1-\frac{1}{p})}} \|f\|_{L^{p}(\R^n)}^{\frac{n-s}{n(1-\frac{1}{p})}}.
\end{equation}
\end{prop}

For $s\in (0,2)$, the Fourier multiplier $(-\Delta)^{s/2}$ is called the \emph{fractional Laplacian of order $s$}. It is a straightforward computation to show that for any Schwartz function $f$,
\begin{equation}
\label{eq:FL}
(-\Delta)^{s/2}f(x) = \underbrace{\frac{2^s\Gamma(\frac{n+s}{2})}{\pi^{n/2}|\Gamma(-\frac{s}{2})|}}_{\eqqcolon C_{d,s}}\PV\int_{\R^n}\frac{f(x)-f(y)}{|x-y|^{n+s}}dy,
\end{equation}
where the notation $\PV$ denotes the principal value of the integral. The RHS of \eqref{eq:FL} evidently converges by the H\"{o}lder continuity for $s\in (0,1)$. For $s\in [1,2)$, we note that by a change of variable,
\begin{equation}
\PV\int_{\R^n} \frac{f(x)-f(y)}{|x-y|^{n+s}}dy = \PV\int_{\R^n} \frac{f(x)-f(x-y)}{|y|^{n+s}}dy = \PV\int_{\R^n} \frac{f(x)-f(x+y)}{|y|^{n+s}}dy.
\end{equation}
Hence, we may symmetrize the RHS of \eqref{eq:FL} so that
\begin{equation}
(-\Delta)^{s/2}f(x) = \frac{C_{d,s}}{2}\PV\int_{\R^n} \frac{2f(x)-f(x-y)-f(x+y)}{|y|^{n+s}}dy,
\end{equation}
which evidently converges by the smoothness of $f$.

Recall from \cref{sec:intro} that the Biot-Savart law for the velocity field $u$ in the $\alpha$-gSQG equation, $\alpha \in [0,2)$, is
\begin{equation}
u=\nabla^\perp (G_\alpha\ast\theta).
\end{equation}
For a Schwartz function $\theta$, we may integrate by parts so that
\begin{equation}
\label{eq:BS_pv}
u = \PV\int_{\R^2}K_\alpha(y)\theta(x-y)dy,
\end{equation}
where $K_\alpha(z)\coloneqq \nabla^\perp G_\alpha(z)$, for nonzero $z\in\R^2$. For $\alpha \in [0,1)$, we can omit the principal value in the RHS of \eqref{eq:BS_pv} since $K_\alpha$ is locally integrable. For $\alpha\in [1,2)$, the Biot-Savart kernel $K_\alpha$ is not locally integrable; however, the principal value integral in the RHS of \eqref{eq:BS_pv} still converges. Indeed, by a standard argument exploiting the oddness of $K_\alpha$, one can show convergence for $\theta\in L^P \cap \dot{C}^\beta$, for $\frac{p(1+\alpha)}{p-1}>2$ and $\beta>\alpha$.

\subsection{Gronwall-Bellman Inequality}
The precise version of the Gronwall-Bellman inequality used in \cref{sec:boot} is the following proposition.

\begin{prop}
\label{prop:GB}
Let $J=[a,b]$, for $-\infty<a<b<\infty$ and let $u,f: J\rightarrow [0,\infty)$ be continuous functions. Let $n: J \rightarrow (0,\infty)$ be a continuous, nondecreasing function. If
\begin{equation}
u(t) \leq n(t) + \int_{a}^t f(s)u(s)ds, \qquad t\in J,
\end{equation}
then
\begin{equation}
u(t)\leq n(t)\exp\paren*{\int_a^t f(s)ds}, \qquad t\in J.
\end{equation}
\end{prop}
\begin{proof}
See Theorem 1.3.1 in \cite{Pachpatte1997}.
\end{proof}

\subsection{Maximum and Minimum Principles for gSQG}
For the inviscid gSQG equation, the $L^p$ norms of classical solutions are conserved. Indeed, suppose that $\theta$ is a classical solution to \eqref{eq:gSQG} with velocity field $u$. By the Cauchy-Lipschitz theorem, the ordinary differential equation
\begin{equation}
\label{eq:Lag}
\begin{cases}
\dot{X}(t,x) &= u(t,X(t,x)) \\
X(0,x) &= x \in \R^n
\end{cases},
\end{equation}
has a unique local solution. Since $u$ is divergence-free, the map $x\mapsto X(t,x)$ is $C^1$ with unit Jacobian, so that by the inverse function theorem, the back-to-labels map $Y(t,x) \coloneqq X(t,x)^{-1}$ exists and is regular. Since
\begin{equation}
\tl{\theta}(t,x) \coloneqq \theta^0(X^{-1}(t,x))
\end{equation}
is a solution to the gSQG equation, the uniqueness of classical solutions implies that $\theta=\tl{\theta}$. Consequently, we see that properties such as $L^p$ norms and sign of the initial datum $\theta^0$ are preserved by the flow.

For the dissipative gSQG equation, the preceding argument breaks down as the Lagrangian equation \eqref{eq:Lag} no longer holds. However, the effect of dissipation is to remove energy from the system leading to the non-growth of $L^p$ norms.

\begin{prop}[Nonincreasing $L^p$ norms]
\label{prop:Lp_ni}
Let $\theta: [0,T)\times\R^2 \rightarrow\R$ be a smooth function and $u: [0,T)\times\R^2\rightarrow\R^2$ be smooth, divergence-free vector field such that that
\begin{equation}
\p_t\theta+u\cdot\nabla\theta+\kappa (-\Delta)^{\gamma/2}\theta=0
\end{equation}
with $\kappa\geq 0$, $\gamma\in [0,2]$. Then for all $1\leq p\leq \infty$, we have that
\begin{equation}
\|\theta(t)\|_{L^p(\R^2)} \leq \|\theta(0)\|_{L^p(\R^2)}.
\end{equation}
\end{prop}
\begin{proof}
See Theorem 3.3 and Corollary 3.4 in \cite{Res95} and Corollary 2.6 in \cite{Cordoba2004}.
\end{proof}

\begin{remark}
It is possible to show that the $L^p$ norms decay with time, yielding an improvement over \cref{prop:Lp_ni}. Since we will not use such decay in this article, we have not presented the result. The interested reader may consult \cite{Cordoba2004}.
\end{remark}

\begin{prop}[Maximum/Minimum principle]
\label{prop:MMP}
Let $\theta: [0,T)\times\R^2 \rightarrow\R$ be a smooth function and $u: [0,T)\times\R^2\rightarrow\R^2$ be smooth, divergence-free vector field such that that
\begin{equation}
\p_t\theta+u\cdot\nabla\theta+\kappa (-\Delta)^{\gamma/2}\theta=0
\end{equation}
with $\kappa\geq 0$, $\gamma\in [0,2]$. Suppose also that there is some $s>1$ such that $\theta(t)\in H^s(\R^2)$ for all $t\in [0,T)$. Define $M(t)\coloneqq \sup_{x\in\R^2} \theta(t,x)$ and $m(t)\coloneqq \inf_{x\in\R^2}\theta(t,x)$ for $t\in [0,T)$. Then $M$ and $m$ are nonincreasing and nondecreasing Lipschitz functions on $[0,T)$, respectively. In particular, if $\theta(0)$ is nonpositive (nonnegative), then $\theta(t)$ is nonpositive (nonnegative) for all $t\in [0,T)$.
\end{prop}
\begin{proof}
See the proof of Theorem 4.1 in \cite{Cordoba2004} and the proof of Theorem 2.1 in \cite{CCCF2005}.
\end{proof}

\section{gSQG Point Vortex Model}
In an effort to make this article self-contained, we briefly discuss in this section some basic properties of the gSQG point vortex model, such as Hamiltonian structure, conserved quantities, and local/global well-posedness. Our discussion is a natural generalization of well-known facts for Euler point vortices (cf. \cite{MP1984ln, MP2012book}); however, our presentation perhaps differs from the existing literature.

\label{sec:PVM}
\subsection{Hamiltonian Structure}
\label{ssec:PVM_Ham}
We first review the Hamiltonian structure of gSQG point vortex model. For more on Hamiltonian systems, we refer to the books \cite{Arnold2007CM, MR2013book}.

Fix an intensity vector $\ua_N\in (\R\setminus\{0\})^N$. We define a symplectic structure on $\R^{2N}$ as follows. For each $i\in \{1,\ldots,N\}$, define a $2\times 2$ matrix $\J_i$ by
\begin{equation}
\J_i \coloneqq
\begin{bmatrix}
0 & -a_i^{-1} \\
a_i^{-1} & 0
\end{bmatrix},
\end{equation}
and define $\J$ to be the $2N\times 2N$ block-diagonal matrix with diagonal entries $\J_i$. Since each $a_i$ is nonzero, it follows that the matrix $\J$ is invertible, and since each $\J_i$ is skew-symmetric, it follows that $\J$ is skew-symmetric. To the matrix $\J$ we associate a form $\omega$ defined by
\begin{equation}
\omega: \R^{2N} \times \R^{2N} \rightarrow \R, \qquad \omega(\ux_N,\uy_N) \coloneqq \J^{-1}\ux_N \cdot \uy_N,
\end{equation}
where $\cdot$ denotes the standard inner product on $\R^{2N}$. Using the skew-symmetry and invertibility of $\J$ together with the Riesz representation theorem, it is easy to check that $\omega$ is a nondegenerate $2$-form, hence a symplectic form.

Next, we define the Hamiltonian functional
\begin{equation}
H_\alpha(\ux_N) \coloneqq \frac{1}{2}\sum_{1\leq i\neq j\leq N} a_i a_j G_\alpha(x_i-x_j), \qquad \ux_N\in \R^{2N}\setminus\Delta_N.
\end{equation}
To show that the gSQG point vortex model \eqref{eq:PVM} is the equation of motion induced by the Hamiltonian $H_\alpha$ with respect to the symplectic form $\omega$, we need to compute the Hamiltonian vector field $\nabla_\omega H_\alpha$, which is the (necessarily unique) vector field on $\R^{2N}\setminus\Delta_N$ satisfying 
\begin{equation}
dH_\alpha[\ux_N](\uy_N) = \omega(\nabla_\omega H_\alpha(\ux_N), \uy_N), \qquad \ux_N,\uy_N \in \R^{2N}\setminus\Delta_N.
\end{equation}
The reader may check from the definition of the gradient vector field $\nabla H_\alpha$ that $\nabla_\omega H_\alpha(\ux_N) = \J\nabla H_\alpha(\ux_N)$. Now it follows from the chain rule that for each $i\in\{1,\ldots,N\}$, the $i^{th}$ component of $\nabla H_\alpha(\ux_N)$ is the 2-vector
\begin{equation}
(\nabla H_\alpha(\ux_N))_i = \sum_{1\leq i\neq j\leq N} a_i a_j (\nabla G_\alpha)(x_i-x_j).
\end{equation}
From the block-diagonal structure of the matrix $\J$, we see that
\begin{align}
(\J\nabla H_\alpha(\ux_N))_i = \J_i(\nabla H_\alpha(\ux_N))_i &= \sum_{1\leq i\neq j\leq N} a_j\begin{bmatrix}
-((\nabla G_\alpha)(x_i-x_j))_2 & ((\nabla G_\alpha)(x_i-x_j))_1 \end{bmatrix} \nn\\
&=\sum_{1\leq i\neq j\leq N} a_j (\nabla^\perp G_\alpha)(x_i-x_j),
\end{align}
which implies that $(\nabla_\omega H_\alpha(\ux_N))_i = \sum_{1\leq i\neq j\leq N}a_j K_\alpha(x_i-x_j)$. Thus, we have shown that equation \eqref{eq:PVM} can be reformulated as
\begin{equation}
\label{eq:PVM_sym}
\dot{\ux}_N(t) = \nabla_\omega H_\alpha(\ux_N(t)),
\end{equation}
which is precisely what it means to be a Hamiltonian equation of motion.

We now discuss conserved quantities for equation \eqref{eq:PVM}. Recall that a symplectic form $\omega$ canonically induces a Poisson structure by
\begin{equation}
\pb{F}{G}(\ux_N) \coloneqq \omega(\nabla_\omega F(\ux_N), \nabla_\omega G(\ux_N)), \qquad \ux_N\in\R^{2N}\setminus\Delta_N,
\end{equation}
for functions $F,G\in C_{loc}^\infty(\R^{2N}\setminus\Delta_N; \R)$. One may check that
\begin{equation}
\pb{F}{G}(\ux_N) = \sum_{i=1}^N -\frac{1}{a_i}\paren*{\frac{\p G}{\p x_{i,2}}\frac{\p F}{\p x_{i,1}} - \frac{\p G}{\p x_{i,1}}\frac{\p F}{\p x_{i,2}}},
\end{equation}
which should remind the reader of Hamilton's equations. It now follows from the formulation \eqref{eq:PVM_sym} that $\ux_N(t)$ is a solution of equation \eqref{eq:PVM} if and only if
\begin{equation}
\label{eq:PVM_P}
\frac{d}{dt}\paren*{F\circ \ux_N}(t) = \pb{H_\alpha}{F}(\ux_N(t)), \qquad \forall F\in C_{loc}^\infty(\R^{2N}\setminus\Delta_N;\R).
\end{equation}
Using the anti-symmetry of the Poisson bracket, we see from \eqref{eq:PVM_P} that the Hamiltonian $H_\alpha$ is trivially conserved by solutions to the point vortex model \eqref{eq:PVM}. Moreover, since the Hamiltonian $H_\alpha$ is invariant with respect to translations and rotations, it follows from Noether's theorem--or direct computation--that solutions conserve the quantities
\begin{align}
M(\ux_N) &\coloneqq \sum_{i=1}^N a_i x_i, \qquad \ux_N\in\R^{2N},\\
I(\ux_N) &\coloneqq \sum_{i=1}^N a_i|x_i|^2, \qquad \ux_N\in\R^{2N},
\end{align}
respectively. The vector-valued quantity $M$, for $\sum_{i=1}^N a_i\neq 0$, is proportional to the \emph{center of vorticity}, and the scalar-valued quantity $I$ is called the \emph{moment of inertia}.

\subsection{Local Well-Posedness}
\label{ssec:PVM_LWP}
We now discuss the well-posedness of the point vortex model \eqref{eq:PVM}, in particular whether for given initial data $(\ua_N,\ux_N^0)$, there exists some time interval on which the equation has a unique solution. We refer to this property as \emph{local well-posedness}. First, we must clarify what we mean by a solution to \eqref{eq:PVM}.

\begin{mydef}[Solution to PVM]
\label{def:pvm_soln}
Fix $\alpha \geq 0$ and an intensity vector $\ua_N\in (\R\setminus\{0\})^N$. Given a time $T>0$ and initial configuration $\ux_{N}^{0}\in\R^{2N} \setminus\Delta_N$, we say that a function $\ux_{N}\in C([0,T];\R^{2N})$ is a solution to equation \eqref{eq:PVM} if
\begin{equation}
\label{eq:sep_con}
\min_{1\leq i< j\leq N} |x_{i}(t)-x_{j}(t)| > 0, \qquad t\in [0,T]
\end{equation}
and
\begin{equation}
\label{eq:PVM_mild}
x_{i}(t) = x_{i}^{0}+\sum_{1\leq i\neq j\leq N} a_{j} \int_{0}^{t} K_{\alpha}(x_{i}(s)-x_{j}(s))ds, \qquad  t\in [0,T]
\end{equation}
for every $i\in\{1,\ldots,N\}$. We say that $\ux_{N}\in C_{loc}([0,T);\R^{2N})$ is a solution to equation \eqref{eq:PVM} if for every $0\leq T'<T$, $\ux_{N}\in C([0,T'];\R^{2N})$ is a solution to \eqref{eq:PVM}. We say that a solution $\ux_N \in C_{loc}([0,T);\R^{2N})$ has \emph{maximal lifespan} if $\ux_N$ is not a solution on the interval $[0,T]$. If $T=\infty$, then we say that the solution is \emph{global}.
\end{mydef}

\begin{remark}
\label{rem:PVM_smooth}
Since for finite $T>0$, the interval $[0,T]$ is compact, it follows from Weierstrass's extreme value theorem that the condition
\begin{equation}
\min_{1\leq i< j\leq N} |x_{i}(t)-x_{j}(t)| > 0, \qquad t\in [0,T]
\end{equation}
implies that
\begin{equation}
\min_{t\in [0,T]} \min_{1\leq i< j\leq N} |x_{i}(t)-x_{j}(t)| > 0.
\end{equation}
Using this property together with the smoothness of $K_\alpha$ away from the origin and induction, it follows that if $\ux_{N}\in C([0,T];\R^{2N})$ is a solution to \eqref{eq:PVM}, then $\ux_{N}\in C^{\infty}([0,T];\R^{2N})$.
\end{remark}

\begin{remark}
\label{rem:PVM_uniq}
If a solution to equation \eqref{eq:PVM} exists, then it is necessarily unique. Indeed, suppose that $\ux_{N},\ul{y}_{N}\in C([0,T];\R^{2N})$ are solutions to \eqref{eq:PVM} with the same initial data $\ux_{N}^{0}$. Let $\psi=\tl{\psi}(|\cdot|)\in C_{c}^{\infty}(\R)$ be a radial function satisfying $0\leq \psi\leq 1$ and
\begin{equation}
\psi(x)=
\begin{cases}
1, & |x|\leq 1\\
0, & |x|\geq 2
\end{cases}.
\end{equation}
Define the function $\psi_{\varepsilon}(x) = \psi(|x|^2/\varepsilon^2)$ and the regularized Green's function
\begin{equation}
\label{eq:GF_reg}
G_{\alpha,\varepsilon}(z) \coloneqq G_{\alpha}(z)\paren*{1-\psi_{\varepsilon}(z)}, \qquad z\in\R^2\setminus\{0\}.
\end{equation}
Since $K_{\alpha,\varepsilon} \coloneqq \nabla^\perp G_{\alpha,\varepsilon}$ is Lipschitz, the regularized Cauchy problem
\begin{equation}
\label{eq:PVM_reg}
\begin{cases}
y_{i,\varepsilon}(t) = \sum_{1\leq j\neq i}^{N}a_{j}K_{\alpha,\varepsilon}(y_{i,\varepsilon}(t)-y_{j,\varepsilon}(t)) \\
y_{i,\varepsilon}(0) = x_{i}^{0}
\end{cases}.
\end{equation}
has a unique solution in the sense of \cref{def:pvm_soln} by the Cauchy-Lipschitz theorem. If $\varepsilon \in (0,\delta/2)$, where
\begin{equation}
\delta \coloneqq \min_{t\in [0,T]} \min_{1\leq i<j\leq N} |x_i(t)-x_j(t)|,
\end{equation}
then $\ux_{N}$ and $\ul{y}_{N}$ are both solutions to \eqref{eq:PVM_reg}, and therefore $\ux_{N}=\ul{y}_{N}$. Thus, in the sequel, we may unambiguously introduce the notation
\begin{equation}
d_T(\ux_N^0) \coloneqq \min_{t\in [0,T]}\min_{1\leq i<j\leq N} |x_i(t)-x_j(t)|
\end{equation}
to denote the minimum distance between any two point vortices on the interval $[0,T]$ starting from the initial configuration $\ux_N^0$.
\end{remark}

We now show that the Cauchy problem for the gSQG point vortex model is locally well-posed. Since the vortices are initially separated, we expect that they should remain separated for short times. Hence, the Biot-Savart kernel $K_\alpha$ is evaluated at a fixed positive distance from the origin, allowing us to use the Cauchy-Lipschitz theorem. Moreover, so long as the trajectories remain at positive distance from one another, we should be able to iterate the local existence argument. It is worth remarking that we are able to prove local well-posedness for arbitrary $\alpha\geq 0$, since the hypersingularity of the Green's function $G_\alpha$ for $\alpha\geq 2$ is avoided per the preceding commentary.

\begin{thm}[Local well-posedness]
\label{thm:PVM_lwp}
Fix $\alpha\geq 0$, $N\in\N$, $\ua_{N}\in (\R\setminus\{0\})^{N}$, and $\ux_{N}^{0}\in\R^{2N} \setminus\Delta_N$. Then there exists a time $T=T(\alpha,N,\ua_{N},d_0(\ux_{N}^0))>0$ and a unique maximal-lifespan solution $\ux_{N}\in C([0,T);\R^{2N})$ to equation \eqref{eq:PVM}. The data-to-solution map $\ux_N^0\mapsto \ux_N$ is continuous. Furthermore, if $T<\infty$, then we have the collapse criterion that there exist distinct $i,j\in\{1,\ldots,N\}$, such that
\begin{equation}
\lim_{t\rightarrow T^{-}} |x_{i}(t)-x_{j}(t)| = 0.
\end{equation}
\end{thm}
\begin{proof}
Fix $N\in\N$, $\ua_{N}\in (\R\setminus\{0\})^{N}$, and $\ux_{N}^{0}\in\R^{2N}\setminus\Delta_N$. Let $\varepsilon>0$ be a parameter, the precise value of which will be specified momentarily. Let $\ux_{N,\varepsilon}$ denote the solution of the regularized Cauchy problem \eqref{eq:PVM_reg}. For each $i\in\{1,\ldots,N\}$, we study the time evolution of the quantity $|x_{i,\varepsilon}(t)-x_{i}^{0}|^{2}$. By the chain rule and using the equation \eqref{eq:PVM_reg}, we see that
\begin{align}
\frac{d}{dt}|x_{i,\varepsilon}(t)-x_{i}^{0}|^{2} &= 2(x_{i,\varepsilon}(t)-x_{i}^{0}) \cdot \sum_{1\leq k\neq i\leq N}a_{k} K_{\alpha,\varepsilon}(x_{i,\varepsilon}(t)-x_{k,\varepsilon}(t)).
\end{align}
Since by the Leibnitz rule and the radiality of $\psi$,
\begin{align}
K_{\alpha,\varepsilon}(z) = -\alpha C_{\alpha}\frac{z^{\perp}}{|z|^{2+\alpha}}\paren*{1-\psi_{\varepsilon}(z)} - \frac{2C_{\alpha}z^\perp}{\varepsilon^2|z|^{\alpha}}\tl{\psi}'\paren*{\frac{|z|^{2}}{\varepsilon^{2}}}, \qquad z\in\R^2\setminus\{0\},
\end{align}
we see from direct estimation that
\begin{equation}
\max_{1\leq i\neq k\leq N} \sup_{t\geq 0} |K_{\alpha,\varepsilon}(x_{i,\varepsilon}(t)-x_{k,\varepsilon}(t))| \lesssim_{\alpha}\frac{1}{\varepsilon^{1+\alpha}}.
\end{equation}
Therefore,
\begin{equation}
\left|2(x_{i,\varepsilon}(t)-x_{i}^{0})\cdot \sum_{1\leq k\neq i\leq N}a_{k} K_{\alpha,\varepsilon}(x_{i,\varepsilon}(t)-x_{k,\varepsilon}(t))\right|  \leq \frac{\tilde{C}_{\alpha}\|\ua_N\|_{\ell^1}|x_{i,\varepsilon}(t)-x_{i}^{0}|}{\varepsilon^{1+\alpha}},
\end{equation}
which implies that
\begin{equation}
|x_{i,\varepsilon}(t)-x_{i}^{0}|^{2} \leq \paren*{\frac{\tilde{C}_{\alpha}\|\ua_N\|_{\ell^1}t}{\varepsilon^{1+\alpha}}}^{2}, \qquad i\in\{1,\ldots,N\}.
\end{equation}
We choose $\varepsilon_0=d_0(\ux_N^0)/16$ and $T>0$ so that
\begin{equation}
\frac{\tilde{C}_{\alpha}\|\ua_N\|_{\ell^1} T}{\varepsilon_0^{1+\alpha}} \leq \frac{d_0(\ux_N^0)}{4}.
\end{equation}
Then by the reverse triangle inequality, for distinct $i,j\in\{1,\ldots,N\}$ and any $t\in [0,T]$, we have the lower bound
\begin{equation}
|x_{i,\varepsilon_{0}}(t)-x_{j,\varepsilon_{0}}(t)| \geq |x_{i}^{0}-x_{j}^{0}| - |x_{i,\varepsilon_{0}}(t)-x_{i}^{0}| - |x_{j,\varepsilon_{0}}-x_{j}^{0}| \geq \frac{|x_{i}^{0}-x_{j}^{0}|}{2}.
\end{equation}
In particular,
\begin{equation}
\min_{0\leq t\leq T}\min_{1\leq i < j\leq N} |x_{i,\varepsilon_{0}}(t)-x_{j,\varepsilon_{0}}(t)|  > 2\varepsilon_{0},
\end{equation}
so that for each $i\in\{1,\ldots,N\}$,
\begin{equation}
\dot{x}_{i,\varepsilon_{0}}(t) = \sum_{1\leq j\neq i\leq N} a_{j}K_{\alpha,\varepsilon_0}(x_{i,\varepsilon_{0}}(t)-x_{j,\varepsilon_{0}}(t)) = \sum_{1\leq j\neq i\leq N} a_{j}K_{\alpha}(x_{i,\varepsilon_{0}}(t)-x_{j,\varepsilon_{0}}(t)).
\end{equation}
By \cref{rem:PVM_uniq}, we can unambiguously define the solution to equation \eqref{eq:PVM} by $\ux_{N} \coloneqq \ux_{N,\varepsilon_{0}}$.

We leave verification of continuous dependence of the initial data to the reader. So it remains for us to prove the collapse criterion. Suppose that $\ux_{N}$ has maximal lifespan $[0,T)$ with $T<\infty$ and that
\begin{equation}
\lim_{t\rightarrow T^{-}} \min_{1\leq i< j\leq N} |x_{i}(t)-x_{j}(t)| = \delta>0,
\end{equation}
for some finite $\delta$. Choose $T'<T$ sufficiently close to $T$ so that
\begin{equation}
\inf_{T'\leq t<T} \min_{1\leq i< j\leq N} |x_{i}(t)-x_{j}(t)| \geq \frac{\delta}{2}
\end{equation}
and $T-T'$ is smaller than the local time $T_{\Delta}$ of existence to equation \cref{eq:PVM} for intensities $\ua_{N}$ and minimal initial separation $\delta/2$. We then use our local existence result shown in the preceding paragraph in order to obtain a solution $\ul{z}_{N}\in C([0,T_{\Delta}];\R^{2N})$ to the Cauchy problem
\begin{equation}
\begin{cases}
\dot{z}_{i}(t) = \sum_{1\leq j\neq i\leq N} a_{j}K_{\alpha}(z_{i}(t)-z_{j}(t)) \\
z_{i}(0) = x_{i}(T')
\end{cases}.
\end{equation}
We now define $\ul{y}_{N}(t)$ piecewise by the formula
\begin{equation}
\ul{y}_{N}(t)=
\begin{cases}
\ux_{N}(t), & 0\leq t\leq T' \\
\ul{z}_{N}(t-T'), & T'<t\leq T'+T_{\Delta}
\end{cases}.
\end{equation}
It is evident that $\ul{y}_{N}$ is continuous, that
\begin{equation}
\min_{0\leq t\leq T'+T_{\Delta}} \min_{1\leq i< j\leq N} |y_{i}(t)-y_{j}(t)| > 0,
\end{equation}
and that
\begin{equation}
\label{eq:y_duh}
y_{i}(t) = x_{i}^{0} + \sum_{1\leq j\neq i\leq N} a_{j} \int_{0}^{t} K_{\alpha}(x_{i}(s)-x_{j}(s))ds, \qquad t\in [0,T'].
\end{equation}
We claim that equation \eqref{eq:y_duh} holds on the strictly larger interval $[0,T'+T_\Delta]$. Indeed, for $t\in (T',T'+T_{\Delta}]$, we have that
\begin{align}
y_{i}(t) = z_{i}(t-T') &= x_{i}(T') + \sum_{1\leq j\neq i\leq N} a_{j} \int_{0}^{t-T'} K_{\alpha}(z_{i}(s)-z_{j}(s))ds \nn\\
&= x_{i}^{0} + \sum_{1\leq j\neq i\leq N} a_{j}\int_{0}^{T'} K_{\alpha}(x_{i}(s)-x_{j}(s))ds + \sum_{1\leq j\neq i\leq N} a_{j} \int_{0}^{t-T'} K_{\alpha}(z_{i}(s)-z_{j}(s))ds \nn\\
&= x_{i}^{0} + \sum_{1\leq j\neq i\leq N} a_{j}\paren*{\int_{0}^{T'} K_{\alpha}(x_{i}(s)-x_{j}(s))ds + \int_{T'}^{t} K_{\alpha}(z_{i}(s-T')-z_{j}(s-T'))ds} \nn\\
&= x_{i}^{0} + \sum_{1\leq j\neq i\leq N} a_{j}\paren*{\int_{0}^{T'} K_{\alpha}(y_{i}(s)-y_{j}(s))ds + \int_{T'}^{t} K_{\alpha}(y_{i}(s)-y_{j}(s))ds} \nn\\
&=x_{i}^{0} + \sum_{1\leq j\neq i\leq N} a_{j}\int_{0}^{t} K_{\alpha}(y_{i}(s)-y_{j}(s))ds,
\end{align}
where we use the equation for $x_{i}$ to obtain the second equality, a change of variable to obtain the third equality, the definition of $y_{i}$ and $y_{j}$ to obtain the fourth equality, and the additivity of the Lebesgue integral to obtain the fifth equality. Therefore, we have shown that $\ul{y}_{N}$ is a solution to equation \eqref{eq:PVM} with initial datum $\ux_{N}^{0}$ and lifespan $T'+T_{\Delta}>T$, which contradicts our assumption that $[0,T)$ is maximal.
\end{proof}

\subsection{Global Well-Posedness}
\label{ssec:PVM_GWP}
Having discussed local well-posedness of the gSQG point vortex model, we end \cref{sec:PVM} with some comments on global well-posedness. For $N$ vortices with intensities all of the same sign (i.e. $\sgn(a_1)=\cdots=\sgn(a_N)$), global existence follows from conservation of the Hamiltonian. Indeed, for $\alpha>0$ and, without loss of generality, $a_i$ all positive, we have that
\begin{equation}
\frac{C_\alpha a_{i_0}a_{j_0}}{|x_{i_0}(t)-x_{j_0}(t)|^\alpha} \leq C_\alpha \sum_{1\leq i\neq j\leq N} \frac{a_i a_j}{|x_i(t)-x_j(t)|^{\alpha}} = H_\alpha(t) = H_\alpha(0),
\end{equation}
for any $1\leq i_0<j_0\leq N$. Hence,
\begin{equation}
|x_{i_0}(t)-x_{j_0}(t)| \geq \paren*{\frac{C_\alpha a_{i_0}a_{j_0}}{H_\alpha(0)}}^{1/\alpha} \geq \paren*{\frac{C_\alpha \min_{1\leq i\neq j\leq N} a_i a_j}{H_\alpha(0)}}^{1/\alpha}.
\end{equation}
Since $t\in [0,T)$ and $1\leq i_0<j_0\leq N$ were arbitrary, we obtain that
\begin{equation}
\label{eq:cs_lb}
\min_{1\leq i<j\leq N} \inf_{t\in[0,T)} |x_{i_0}(t)-x_{j_0}(t)| \geq \paren*{\frac{C_\alpha \min_{1\leq i\neq j\leq N} a_i a_j}{H_\alpha(0)}}^{1/\alpha}.
\end{equation}
From the collapse criterion of \cref{thm:PVM_lwp}, we conclude that $T=\infty$. The $\alpha=0$ case is slightly more involved due to the fact that the logarithm changes sign on $(0,\infty)$; however, one can show a lower bound similar to \eqref{eq:cs_lb} by combining conservation of the Hamiltonian $H_0$ with conservation of the moment of inertia $I$. We leave the details to the reader.

The system of $N=2$ vortices has an explicit global solution: either the two vortices rotate around their center of vorticity or they travel parallel to each other, maintaining constant separation. For $N\geq 3$ vortices with intensities of not identical sign, we recall from the introduction that finite-time collapse of three vortices is possible. For Euler point vortices (see Chapter 4 of \cite{MP2012book}), the classical counter-example to global existence is the evolution of the initial data
\begin{equation}
\label{eq:ID_c}
a_1=a_2=2, \enspace a_3=-1 \text{ and } x_1=(-1,0), x_2=(1,0), x_3=(1,\sqrt{2}),
\end{equation}
which spirals in a self-similar fashion to a single point. Moreover, this counter-example carries over to gSQG point vortices for $\alpha>0$ \cite{BB2018}. Nevertheless, the initial data leading to collapse, such as that given by \eqref{eq:ID_c}, are exceptional. Indeed, if the intensities $a_1,\ldots,a_N$ satisfy a "generic" condition which ensures that point vortices cannot shoot off to $\infty$ in finite time, then by proving an integral inequality for the control function
\begin{equation}
\Phi(\ux_N) \coloneqq \sum_{1\leq i\neq j\leq N} \ln|x_i-x_j|,
\end{equation}
which is singular when two vortices collapse yet still locally integrable, one can show (see for instance Theorem 2.2 in \cite{MP2012book}) using Chebyshev's inequality that the $2N$-dimensional Lebesgue measure of initial data leading to collapse is zero. The argument carries over without difficulty for mSQG point vortices due to the fact that $|x-y|^{-s}$ is locally integrable for any $s<2$, but it breaks down for SQG and more singular gSQG point vortices. The reader can find the details in \cite{GRmSQGPV2018}. We summarize the preceding discussion with the following proposition.

\begin{prop}[Almost sure global existence]
Let $\alpha \in [0,1)$ and $N\in\N$. Suppose that $\ua_N\in (\R\setminus\{0\})^N$ satisfies the condition
\begin{equation}
\sum_{i\in P} a_i \neq 0, \qquad \forall P\subset \{1,\ldots,N\}.
\end{equation}
Then there exists a set $Z\subset \R^{2N}$ with zero $2N$-dimensional Lebesgue measure, such that for every $\ux_N^0 \in \R^{2N}\setminus Z$, equation \eqref{eq:PVM} has a unique global solution.
\end{prop}

\section{Moment of Inertia Integral Inequality}
\label{sec:main_est}
In this section, we prove our main estimate for the approximate moment of inertia used to measure the concentration of the scalar around the time-evolved point vortices $x_1(t),\ldots,x_N(t)$. We treat the inviscid and dissipative cases simultaneously.

\subsection{Setup}
Let $\phi\in C_{c}^{\infty}(\R^{2})$ be a radial bump function such that $0\leq \phi\leq 1$,
\begin{equation}
\phi(x) = 
\begin{cases}
1, & |x|\leq 1 \\
0, & |x| \geq 2
\end{cases},
\end{equation}
and $\phi$ is positive and strictly decreasing for $|x| \in (1,2)$. For $R>0$, set $\phi_{R}(x) \coloneqq \phi(R^{-1}x)$, and for $\epsilon>0$, define the approximate moment of inertia
\begin{equation}
I_{R,\epsilon}(t) \coloneqq \sum_{i=1}^{N}\int_{\R^{2}}|x-x_{i}(t)|^{2}\phi_{R}(x-x_{i}(t))\theta_{\epsilon}(t,x)dx.
\end{equation}
We require that $R\leq d_T(\ux_N^0)/100$. Now associate to $I_{R,\ep}$ the maximal function
\begin{equation}
\label{eq:Imax_def}
\bar{I}_{R,\ep}(t) \coloneqq \sup_{0\leq s\leq t} I_{R,\ep}(s).
\end{equation}
The goal of this section is to prove the following proposition.

\begin{prop}[Estimate for $\bar{I}_{R,\ep}$]
\label{prop:Ibar}
Let $\alpha \in [0,2)$ and $\gamma \in [0,2]$. Then there exists a constant $C(\alpha,\gamma)>0$, such that if we define
\begin{equation}
\label{eq:rho_def}
\rho_{R,\ep}(t) \coloneqq \frac{\|\ua_N\|_{\ell^1}\bar{I}_{R,\ep}(t)}{R^{4+\alpha}} + \frac{C_2^{\alpha/2}\bar{I}_{R,\ep}(t)^{1+\frac{2-\alpha}{2}}}{\ep^\alpha R^{6-\alpha}}  + \frac{\bar{I}_{R,\ep}(t)}{R^3 d_T(\ux_N^0)^{1+\alpha}} + \frac{\kappa(\ep)N\|\ua_N\|_{\ell^1}}{R^{\gamma}},
\end{equation}
then
\begin{equation}\label{eq:I_ineq}
\begin{split}
\bar{I}_{R,\epsilon}(t) &\leq I_{R,\ep}(0) + \kappa(\ep)N R^{2-\gamma}\|\ua_N\|_{\ell^1} t + \frac{\|\ua_N\|_{\ell^1}}{d_T(\ux_N^0)^{2+\alpha}}\int_0^t \bar{I}_{R,\ep}(s)ds \\
&\ph{=} + \paren*{\frac{C(\alpha,\gamma)\|\ua_N\|_{\ell^1}}{R^\alpha}+ \frac{C(\alpha,\gamma)\|\ua_N\|_{\ell^1}}{d_T(\ux_N^0)^{1+\alpha}}}\int_0^t s\rho_{R,\ep}(s)ds \\
&\ph{=} + \frac{C_2^{\alpha/2} C(\alpha,\gamma)^{1+\frac{2-\alpha}{2}}}{\ep^\alpha} \int_0^t s^{1+\frac{2-\alpha}{2}}\rho_{R,\ep}(s)^{1+\frac{2-\alpha}{2}}ds
\end{split}
\end{equation}
for $t\in [0,T]$.
\end{prop}

To prove \cref{prop:Ibar}, we proceed by the energy method. Multiplying by $-1$ if necessary, we may assume without loss of generality that $a_{i}>0$ for every $i\in\{1,\ldots,N\}$. Differentiating $I_{R,\epsilon}$ with respect to time and using the chain rule together with the gSQG equation \eqref{eq:gSQG} and the gSQG point vortex model \eqref{eq:PVM}, we see that
\begin{align}
\dot{I}_{R,\epsilon}(t) &= \sum_{i=1}^{N} -2\int_{\R^{2}}(x-x_{i}(t))\cdot \dot{x}_{i}(t)\phi_{R}(x-x_{i}(t))\theta_{\epsilon}(t,x)dx \nn\\
&\phantom{=} \hspace{10mm} -\int_{\R^{2}}|x-x_{i}(t)|^{2}\nabla\phi_{R}(x-x_{i}(t))\cdot \dot{x}_{i}(t)\theta_{\epsilon}(t,x)dx \nn\\
&\phantom{=} \hspace{10mm} + \int_{\R^{2}}|x-x_{i}(t)|^{2}\phi_{R}(x-x_{i}(t))\p_{t}\theta_{\epsilon}(t,x) \nn\\
&= -2\sum_{i=1}^{N}\sum_{1\leq j\neq i\leq N} a_{j}\int_{\R^{2}} (x-x_{i}(t))\cdot K_{\alpha}(x_{i}(t)-x_{j}(t)) \phi_{R}(x-x_{i}(t))\theta_{\epsilon}(t,x)dx \nn\\
&\phantom{=} -\sum_{i=1}^{N}\sum_{1\leq j\neq i\leq N} a_{j}\int_{\R^{2}}|x-x_{i}(t)|^{2}\nabla\phi_{R}(x-x_{i}(t))\cdot K_{\alpha}(x_{i}(t)-x_{j}(t))\theta_{\epsilon}(t,x)dx \nn\\
&\phantom{=} + 2\sum_{i=1}^{N}\int_{\R^{2}} (x-x_{i}(t))\cdot \phi_{R}(x-x_{i}(t)) \frac{\nabla^{\perp}}{|\nabla|^{2-\alpha}}(\theta_{\epsilon})(t,x)\theta_{\epsilon}(t,x)dx \nn\\
&\phantom{=} + \sum_{i=1}^{N}\int_{\R^{2}} |x-x_{i}(t)|^{2}\nabla\phi_{R}(x-x_{i}(t))\cdot \frac{\nabla^{\perp}}{|\nabla|^{2-\alpha}}(\theta_{\epsilon})(t,x)\theta_{\epsilon}(t,x)dx \nn\\
&\ph{=} - \kappa(\ep)\sum_{i=1}^N \int_{\R^2} |x-x_i(t)|^2\phi_R(x-x_i(t)) (-\Delta)^{\gamma/2}(\theta_\ep)(t,x)dx,
\label{eq:sup_RHS}
\end{align}
where we integrate by parts to obtain the ultimate equality. We split the RHS of \eqref{eq:sup_RHS} into the sum of three groups defined as follows:
\begin{align}
\mathrm{Group}_{1} &\coloneqq 2\sum_{i=1}^{N}\int_{\R^2} (x-x_{i}(t))\cdot\phi_{R}(x-x_{i}(t))\theta_{\epsilon}(t,x)\paren*{\frac{\nabla^{\perp}}{|\nabla|^{2-\alpha}}(\theta_{\epsilon})(t,x) - \sum_{1\leq j\neq i\leq N} a_{j}K_{\alpha}(x_{i}(t)-x_{j}(t))}dx \label{eq:MR_G1}\\
\mathrm{Group}_{2} &\coloneqq \sum_{i=1}^{N}\int_{\R^{2}}|x-x_{i}(t)|^{2}\nabla\phi_{R}(x-x_{i}(t))\cdot\theta_\ep(t,x)\paren*{\frac{\nabla^{\perp}}{|\nabla|^{2-\alpha}}(\theta_{\epsilon})(t,x)-\sum_{1\leq j\neq i\leq N} a_{j}K_{\alpha}(x_{i}(t)-x_{j}(t))}dx \label{eq:MR_G2} \\
\mathrm{Group}_3 &\coloneqq -\kappa(\ep)\sum_{i=1}^N \int_{\R^2} |x-x_i(t)|^2 \phi_R(x-x_i(t)) (-\Delta)^{\gamma/2}(\theta_\ep)(t,x)dx \label{eq:MR_G3}.
\end{align}
$\mathrm{Group}_1$ and $\mathrm{Group}_2$ reflect the error between the infinite-dimensional continuum dynamics and the finite-dimensional point-vortex dynamics, while $\mathrm{Group}_3$ is the change in the moment of inertia due to dissipation. Since the analysis for $\mathrm{Group}_{2}$ is similar to the analysis for $\mathrm{Group}_{1}$, we only present the details of the estimates for $\mathrm{Group}_{1}$ and $\mathrm{Group}_3$.

\subsection{Estimate for $\mathrm{Group}_1$}
We begin with estimating $\mathrm{Group}_1$. We first use the integral kernel for the operator $\frac{\nabla^{\perp}}{|\nabla|^{2-\alpha}}$ to write
\begin{equation}\label{eq:dcomp_sub}
\begin{split}
&\int_{\R^{2}} (x-x_{i}(t))\cdot\phi_{R}(x-x_{i}(t))\theta_{\epsilon}(t,x)\frac{\nabla^{\perp}}{|\nabla|^{2-\alpha}}(\theta_{\epsilon})(t,x)dx \\
&\ph{=}= \PV\int_{\R^{2}\times\R^2}\theta_{\epsilon}(t,x)\theta_{\epsilon}(t,y) (x-x_{i}(t))\cdot\phi_{R}(x-x_{i}(t)) K_{\alpha}(x-y)dxdy.
\end{split}
\end{equation}
We now decompose the function $\theta_\ep$ by
\begin{equation}\label{eq:th_dcomp}
\theta_{\epsilon}(t,y) = \sum_{j=1}^{N}\phi_{R}(y-x_{j}(t))\theta_{\epsilon}(t,y) + \underbrace{\paren*{1-\sum_{j=1}^{N}\phi_{R}(y-x_{i}(t))}}_{\eqqcolon \psi_{R}(t,y)}\theta_{\epsilon}(t,y)
\end{equation}
and substitute this decomposition into the RHS of \eqref{eq:dcomp_sub}, leading to the sum
\begin{equation}
\begin{split}
&\sum_{j=1}^{N} \PV\int_{\R^2\times\R^2} \theta_{\epsilon}(t,x)\theta_{\epsilon}(t,y)\phi_{R}(x-x_{i}(t))\phi_{R}(y-x_{j}(t))(x-x_i(t))\cdot K_{\alpha}(x-y)dxdy \\
&\phantom{=} + \PV\int_{\R^2\times\R^2}\theta_{\epsilon}(t,x)\theta_{\epsilon}(t,y)\phi_{R}(x-x_{i}(t))\psi_{R}(t,y)(x-x_i(t))\cdot K_{\alpha}(x-y)dxdy.
\end{split}
\end{equation}
Consider the contribution of the first term in the preceding sum. Observe from the anti-symmetry of the Biot-Savart kernel $K_{\alpha}$ that
\begin{equation}
\begin{split}
&\sum_{i,j=1}^{N}\PV\int_{\R^2\times\R^2} \theta_{\epsilon}(t,x)\theta_{\epsilon}(t,y)\phi_{R}(x-x_{i}(t))\phi_{R}(y-x_{j}(t))(x-x_i(t))\cdot K_{\alpha}(x-y)dxdy \\
&\phantom{=} = \frac{1}{2}\sum_{i=1}^N \PV \int_{\R^2\times\R^2} \theta_{\ep}(t,x)\theta_\ep(t,y)\phi_R(x-x_i(t))\phi_R(y-x_i(t)) (x-y)\cdot K_\alpha(x-y)dxdy \\
&\ph{=}\quad +\sum_{1\leq i\neq j\leq N} \PV\int_{\R^2\times\R^2} \theta_{\epsilon}(t,x)\theta_{\epsilon}(t,y)\phi_{R}(x-x_{i}(t))\phi_{R}(y-x_{j}(t)) (x-x_i(t))\cdot K_{\alpha}(x-y)dxdy.
\end{split}
\end{equation}
Since $K_{\alpha}(x-y)$ is orthogonal to $(x-y)$ for all distinct $x,y\in\R^2$, the first term on the RHS of the preceding equality is zero. Returning to the definition \eqref{eq:MR_G1} of $\mathrm{Group}_{1}$ and substituting into it all of our decompositions and simplifications, we need to estimate the following expressions:
\begin{align}
\mathrm{Group}_{1,1} &\coloneqq 2\sum_{1\leq i\neq j \leq N} \PV\int_{\R^2\times\R^2} \theta_{\epsilon}(t,x)\theta_{\epsilon}(t,y)\phi_{R}(x-x_{i}(t))\phi_{R}(y-x_{j}(t))(x-x_{i}(t)) \nn\\
&\ph{=}\hspace{45mm} \cdot [K_{\alpha}(x-y)- K_{\alpha}(x_{i}(t)-x_{j}(t))]dxdy, \\
\mathrm{Group}_{1,2} &\coloneqq 2\sum_{i=1}^{N} \PV\int_{\R^2\times\R^2}\theta_{\epsilon}(t,x)\theta_{\epsilon}(t,y)\phi_{R}(x-x_{i}(t))\psi_{R}(t,y)(x-x_{i}(t))\cdot K_{\alpha}(x-y)dxdy \\
\mathrm{Group}_{1,3} &\coloneqq - \sum_{1\leq i\neq j \leq N} \int_{\R^2}\theta_\ep(t,x)\phi_{R}(x-x_{i}(t))(x-x_{i}(t))\cdot K_{\alpha}(x_{i}(t)-x_{j}(t)) \nn\\
&\ph{=}\hspace{45mm} \paren*{a_j-\int_{\R^2}\theta_\ep(t,y)\phi_{R}(y-x_{j}(t))dy}dx.
\end{align}
We estimate each of the $\mathrm{Group}_{1,k}$, for $k=1,2,3$, separately.

\subsubsection{Estimate for $\mathrm{Group}_{1,3}$}
We first estimate $\mathrm{Group}_{1,3}$. Using that $|x-x_{i}(t)|\lesssim R$ for $x\in \supp(\tau_{x_i(t)}\phi_R)$ together with the trivial size estimate for the Biot-Savart kernel $K_\alpha$, we see from the Fubini-Tonelli theorem that
\begin{align}
&|\mathrm{Group}_{1,3}| \nn\\
&\phantom{=} \leq \sum_{1\leq i\neq j\leq N} \paren*{\int_{\R^2}\theta_{\epsilon}(t,x)\phi_{R}(x-x_{i}(t))|x-x_{i}(t)| |K_{\alpha}(x_{i}(t)-x_{j}(t))|dx} \left|a_j-\int_{\R^2}\phi_{R}(y-x_{j}(t))\theta_{\epsilon}(t,y)dy\right| \nn\\
&\phantom{=}\lesssim_{\alpha} \sum_{j=1}^N \frac{R}{d_{T}(\ux_{N}^{0})^{1+\alpha}} \paren*{\int_{\R^{2}} \theta_{\epsilon}(t,x)dx}\left|a_j-\int_{\R^2}\phi_R(x-x_j(t))\theta_{\epsilon}(t,y)dy\right| \nn\\
&\phantom{=} \leq \frac{\|\ua_N\|_{\ell^1}R\tl{\mu}_{R,\epsilon}(t)}{d_{T}(\ux_{N}^{0})^{1+\alpha}} \label{eq:MI_G13}
\end{align}
where to obtain the penultimate inequality we have also used the fact that the sets $\supp(\tau_{x_i(t)}\phi_R)$ are pairwise disjoint (see \eqref{eq:tau_pd}) and to obtain the ultimate inequality, we have used $\|\theta(t)\|_{L^1(\R^2)} = \|\ua_N\|_{\ell^1}$ and have introduced the notation
\begin{equation}
\label{eq:tlmu_def}
\tl{\mu}_{R,\epsilon}(t) \coloneqq \sum_{j=1}^{N} |\tl{\mu}_{j,R,\epsilon}(t)|, \qquad \tl{\mu}_{j,R,\epsilon}(t) \coloneqq a_j - \int_{\R^{2}}\phi_{R}(x-x_{j}(t))\theta_{\epsilon}(t,x)dx,
\end{equation}
This last inequality completes the estimate for $\mathrm{Group}_{1,3}$ since we will estimate $\tl{\mu}_{R,\ep}(t)$ with \cref{lem:mu_est}.

\subsubsection{Estimate for $\mathrm{Group}_{1,2}$}
Next, we estimate $\mathrm{Group}_{1,2}$. Observe that since
\begin{equation}
\supp(\tau_{x_j(t)}\phi_R) \subset B(x_{j}(t),2R)
\end{equation}
and $R\ll d_{T}(\ux_{N}^{0})$, we have that
\begin{equation}
\label{eq:tau_pd}
\supp(\tau_{x_j(t)}\phi_{R}) \cap \supp(\tau_{x_j'(t)}\phi_R) = \emptyset, \qquad 1\leq j\neq j'\leq N.
\end{equation}
Furthermore, since $\tau_{x_j(t)}\phi_R \equiv 1$ on the ball $B(x_{j}(t),R)$, it follows that
\begin{equation}
\label{eq:supp_psi}
\supp(\psi_{R}(t,\cdot)) \subset \R^{2}\setminus \bigcup_{j=1}^{N} B(x_{j}(t),R), \qquad t\in [0,\infty)
\end{equation}
and
\begin{equation}
\label{eq:1_psi}
\psi_R(t,x) =1, \qquad x\notin \bigcup_{j=1}^N B(x_j(t),2R), \enspace t\in [0,\infty).
\end{equation}
Now using the anti-symmetry of the Biot-Savart kernel $K_{\alpha}$, we may decompose
\begin{equation}
\mathrm{Group}_{1,2} = \mathrm{Group}_{1,2,1} + \mathrm{Group}_{1,2,2},
\end{equation}
where
\begin{align}
\mathrm{Group}_{1,2,1} &\coloneqq \sum_{i=1}^{N}\int_{|x-y|>\frac{R}{4}}\theta_{\epsilon}(t,x)\theta_{\epsilon}(t,y)\phi_{R}(x-x_{i}(t))\psi_{R}(t,y)(x-x_{i}(t))\cdot K_{\alpha}(x-y)dxdy \\
\mathrm{Group}_{1,2,2} &\coloneqq \frac{1}{2}\sum_{i=1}^{N}\PV\int_{|x-y|\leq \frac{R}{4}}\theta_{\epsilon}(t,x)\theta_{\epsilon}(t,y)K_{\alpha}(x-y) \nn\\
&\ph{=}\quad  \hspace{15mm} \cdot [\phi_{R}(x-x_{i}(t))\psi_{R}(t,y)(x-x_{i}(t)) - \phi_{R}(y-x_{i}(t))\psi_{R}(t,x)(y-x_{i}(t))] dxdy.
\end{align}
We consider $\mathrm{Group}_{1,2,1}$ and $\mathrm{Group}_{1,2,2}$ separately.

\begin{itemize}[leftmargin=*]
\item
The term $\mathrm{Group}_{1,2,1}$ is easy to estimate since we are localized to distance $\gtrsim R$ away from the singularity of the Biot-Savart kernel $K_{\alpha}$. Indeed, introduce a cutoff function
\begin{equation}
\varphi_R(x) \coloneqq 1-\phi\paren*{\frac{4x}{R}}, \qquad x\in\R^2.
\end{equation}
Using the trivial size estimate for $K_{\alpha}$ together with the support observation \eqref{eq:supp_psi} and the fact that $\varphi_R\equiv 1$ in the region $|x|\geq \frac{R}{2}$, we find that
\begin{equation}
|\mathrm{Group}_{1,2,1}| \lesssim_{\alpha} \sum_{i=1}^{N} \frac{\mu_{R,\epsilon}(t)}{R^{\alpha}}\int_{\R^{2}}\theta_{\epsilon}(t,x)\phi_{R}(x-x_{i}(t))dx \leq \frac{\|\ua_N\|_{\ell^1}\mu_{R,\epsilon}(t)}{R^{\alpha}},
\end{equation}
where the ultimate inequality follows from $\|\theta(t)\|_{L^1(\R^2)} = \|\ua_N\|_{\ell^1}$ and we have introduced the notation
\begin{equation}
\label{eq:mu_def}
\mu_{R,\ep}(t) \coloneqq \sum_{j=1}^N \mu_{j,R,\ep}(t), \qquad \mu_{j,R,\ep}(t) \coloneqq \int_{\R^2}\varphi_R(x-x_j(t))\theta_\ep(t,x)dx, \qquad j\in\{1,\ldots,N\}.
\end{equation}
In words, $\mu_{j,R,\ep}$ measures the density of the scalar $\theta_\ep$ at time $t$ outside a ball of radius $\sim R$ centered at $x_{j}(t)$.
\item
To estimate $\mathrm{Group}_{1,2,2}$, we first make the decomposition
\begin{equation}
\label{eq:G_122_d}
\begin{split}
&\phi_{R}(x-x_{i}(t))\psi_{R}(t,y)(x-x_{i}(t)) - \phi_{R}(y-x_{i}(t))\psi_{R}(t,x)(y-x_{i}(t)) \\
&\phantom{=} = [\phi_{R}(x-x_{i}(t))-\phi_{R}(y-x_{i}(t))]\psi_{R}(t,y)(x-x_{i}(t)) \\
&\phantom{=}\quad + [\psi_{R}(t,y)-\psi_{R}(t,x)]\phi_{R}(y-x_{i}(t))(x-x_{i}(t)) \\
&\phantom{=}\quad + \paren*{x-y}\phi_{R}(y-x_{i}(t))\psi_{R}(t,x).
\end{split}
\end{equation}
We substitute this decomposition into the integrand of $\mathrm{Group}_{1,2,2}$ and use the linearity of the Lebesgue integral to obtain the further decomposition
\begin{equation}
\mathrm{Group}_{1,2,2} = \mathrm{Group}_{1,2,2,1} + \mathrm{Group}_{1,2,2,2} + \mathrm{Group}_{1,2,2,3},
\end{equation}
where the term $\mathrm{Group}_{1,2,2,k}$ corresponds to the contribution of the $k^{th}$ term in the RHS of \eqref{eq:G_122_d}, for $k=1,2,3$. Note that if $|x-y|\leq R/4$, then since $y\in\supp (\psi_{R}(t))$ satisfies $|y-x_{i}(t)|\geq R$ (recall \eqref{eq:supp_psi}), we have by the reverse triangle inequality that $|x-x_{i}(t)|\geq 3R/4$. Now using the mean value theorem, we may estimate
\begin{equation}
\label{eq:RHS_bd}
|\mathrm{Group}_{1,2,2,1}| \lesssim_{\alpha} \sum_{i=1}^{N}\int_{{|x-y|\leq \frac{R}{4}}} \frac{\theta_{\epsilon}(t,x)\theta_{\epsilon}(t,y)}{|x-y|^\alpha} 1_{2R\geq |x-x_{i}(t)|\geq \frac{3R}{4}}(x) 1_{|y-x_{i}(t)|\geq \frac{3R}{4}}(y)dxdy.
\end{equation}
Now if $|y-x_i(t)| \leq 3R/4$ and $|x-y|\leq R/4$, then by the triangle inequality, $|x-x_i(t)|\leq R$. Hence by \eqref{eq:supp_psi},
\begin{equation}
[\psi_R(t,y)-\psi_R(t,x)]\phi_R(y-x_i(t))(x-x_i(t))=0.
\end{equation}
It then follows by application of the mean value theorem that $|\mathrm{Group}_{1,2,2,2}|$ is majorized by the RHS of inequality \eqref{eq:RHS_bd}. Lastly, $\mathrm{Group}_{1,2,2,3} = 0$ since $(x-y)$ is orthogonal to $K_\alpha(x-y)$. Applying \cref{prop:Linf_RP} with $(p,n,s)=(\infty,2,2-\alpha)$ together with the fact that $\|\theta(t)\|_{L^\infty(\R^2)} \leq C_2/\ep^2$, we see from the embedding $\ell^1\subset \ell^{1+\frac{2-\alpha}{2}}$ that the RHS of \eqref{eq:RHS_bd} is $\lesssim_\alpha$
\begin{align}
\frac{C_2^{\alpha/2}}{\ep^\alpha} \sum_{i=1}^{N}\paren*{\int_{2R\geq |x-x_{i}(t)|\geq\frac{3R}{4}}\theta_{\epsilon}(t,x)dx}^{1+\frac{2-\alpha}{2}} &\leq \frac{C_2^{\alpha/2}\mu_{R,\epsilon}(t)^{1+\frac{2-\alpha}{2}}}{\epsilon^{\alpha}}.
\end{align}
\end{itemize}

Collecting our estimates for $\mathrm{Group}_{1,2,1}$ and $\mathrm{Group}_{1,2,2}$, we have shown that
\begin{equation}
\label{eq:MI_G12}
|\mathrm{Group}_{1,2}| \lesssim_{\alpha} \frac{\|\ua_N\|_{\ell^1}\mu_{R,\ep}(t)}{R^\alpha} +  \frac{C_2^{\alpha/2}\mu_{R,\epsilon}(t)^{1+\frac{2-\alpha}{2}}}{\epsilon^{\alpha}}.
\end{equation}
This last estimate concludes our analysis for $\mathrm{Group}_{1,2}$ as we will estimate $\mu_{R,\ep}$ with \cref{lem:mu_est} below.

\subsubsection{Estimate for $\mathrm{Group}_{1,1}$}
Since $R\ll d_{T}(\ux_{N}^{0})$ by assumption, it follows from the mean-value theorem and the (reverse) triangle inequality that for each $(x,y)\in \supp(\tau_{x_i(t)}\phi_R) \times \supp(\tau_{x_j(t)}\phi_R)$, there exists $s\in (0,1)$ such that
\begin{align}
|K_{\alpha}(x-y)-K_{\alpha}(x_{i}(t)-x_{j}(t))| &\lesssim_{\alpha} \frac{\paren*{|x-x_{i}(t)| + |y-x_{j}(t)|}}{|s(x-y) + (1-s)(x_{i}(t)-x_{j}(t))|^{2+\alpha}} \nn \\
&\lesssim_{\alpha} \frac{\paren*{|x-x_{i}(t)| + |y-x_{j}(t)|}}{d_T(\ux_N^0)^{2+\alpha}}.
\end{align}
Using this simple estimate, we obtain that
\begin{equation}
\begin{split}
|\mathrm{Group}_{1,1}| &\lesssim_{\alpha} \frac{1}{d_{T}(\ux_{N}^{0})^{2+\alpha}}\sum_{1\leq i\neq j\leq N}  \int_{\R^2\times \R^2} \theta_{\epsilon}(t,x)\theta_{\epsilon}(t,y)\phi_{R}(x-x_{i}(t))\phi_{R}(y-x_{j}(t)) |x-x_{i}(t)| \\
&\ph{=}\hspace{50mm} \paren*{|x-x_{i}(t)|+|y-x_{j}(t)|}dxdy.
\end{split}
\end{equation}
By the Fubini-Tonelli theorem and linearity of the integral,
\begin{equation}
\begin{split}
&\int_{\R^2\times \R^2}\theta_{\epsilon}(t,x)\theta_{\epsilon}(t,y)\phi_{R}(x-x_{i}(t))\phi_{R}(y-x_{j}(t)) |x-x_{i}(t)|\paren*{|x-x_{i}(t)|+|y-x_{j}(t)|}dxdy \\
&\ph{=} = \underbrace{\paren*{\int_{\R^2}|x-x_i(t)|^2\phi_R(x-x_i(t))\theta_\ep(t,x)dx}\paren*{\int_{\R^2}\phi_R(y-x_j(t))\theta_\ep(t,y)dy}}_{\eqqcolon \mathrm{Term}_{1,(i,j)}} \\
&\ph{=}\quad + \underbrace{\paren*{\int_{\R^2}|x-x_i(t)|\phi_R(x-x_i(t))\theta_\ep(t,x)dx}\paren*{\int_{\R^2}|y-x_j(t)|\phi_R(y-x_j(t))\theta_\ep(t,y)dy}}_{\eqqcolon \mathrm{Term}_{2,(i,j)}}.
\end{split}
\end{equation}
Using the definition of $I_{R,\ep}$, the pairwise disjointness of the sets $\supp(\tau_{x_j(t)}\phi_R)$ and $\|\theta_\ep(t)\|_{L^1(\R^2)}=\|\ua_N\|_{\ell^1}$, we see that
\begin{equation}
\sum_{1\leq i\neq j\leq N} \mathrm{Term}_{1,(i,j)} \lesssim \|\theta_{\ep}^0\|_{L^1(\R^2)} \sum_{i=1}^N \int_{\R^2}|x-x_i(t)|^2\phi_R(x-x_i(t))\theta_\ep(t,x)dx \leq \|\ua_N\|_{\ell^1} I_{R,\ep}(t).
\end{equation}
By Cauchy-Schwarz and similar arguments,
\begin{equation}
\sum_{1\leq i\neq j\leq N} \mathrm{Term}_{2,(i,j)} \lesssim_\alpha \|\ua_N\|_{\ell^1} I_{R,\epsilon}(t).
\end{equation}
Collecting our estimates, we conclude that
\begin{equation}
\label{eq:MI_G11}
|\mathrm{Group}_{1,1}| \lesssim_\alpha \frac{\|\ua_N\|_{\ell^1}I_{R,\epsilon}(t)}{d_T(\ux_N^0)^{2+\alpha}} .
\end{equation}

\subsubsection{Final Estimate}
Combining the estimates \eqref{eq:MI_G11}, \eqref{eq:MI_G12}, and \eqref{eq:MI_G13} for $\mathrm{Group}_{1,1}$, $\mathrm{Group}_{1,2}$, and $\mathrm{Group}_{1,3}$, respectively, we obtain the final $\mathrm{Group}_1$ estimate
\begin{equation}
\label{eq:MI_G1}
|\mathrm{Group}_1| \lesssim_\alpha  \frac{\|\ua_N\|_{\ell^1}I_{R,\epsilon}(t)}{d_T(\ux_N^0)^{2+\alpha}}  + \frac{\|\ua_N\|_{\ell^1}\mu_{R,\ep}(t)}{R^\alpha} +  \frac{C_2^{\alpha/2}\mu_{R,\epsilon}(t)^{1+\frac{2-\alpha}{2}}}{\epsilon^{\alpha}} + \frac{R\|\ua_N\|_{\ell^1}\tl{\mu}_{R,\epsilon}(t)}{d_{T}(\ux_{N}^{0})^{1+\alpha}}. 
\end{equation}
By similar arguments, we also obtain the $\mathrm{Group}_2$ estimate
\begin{equation}
\label{eq:MI_G2}
|\mathrm{Group}_2| \lesssim_\alpha \frac{\|\ua_N\|_{\ell^1}I_{R,\epsilon}(t)}{d_T(\ux_N^0)^{2+\alpha}}  + \frac{\|\ua_N\|_{\ell^1}\mu_{R,\ep}(t)}{R^\alpha} +  \frac{C_2^{\alpha/2}\mu_{R,\epsilon}(t)^{1+\frac{2-\alpha}{2}}}{\epsilon^{\alpha}} + \frac{R\|\ua_N\|_{\ell^1}\tl{\mu}_{R,\epsilon}(t)}{d_{T}(\ux_{N}^{0})^{1+\alpha}}.
\end{equation}

\subsection{Estimate for $\mathrm{Group}_3$}
We only present the details for the nonlocal case $\gamma \in (0,2)$, as it is more technically involved than the local case $\gamma\in\{0,2\}$. By Plancherel's theorem and translation invariance of the fractional Laplacian,
\begin{equation}
\begin{split}
\int_{\R^2}|x-x_{i}(t)|^{2}\phi_{R}(x-x_{i}(t))(-\Delta)^{\gamma/2}(\theta_{\epsilon})(t,x)dx = \int_{\R^2}(-\Delta)^{\gamma/2}\paren*{|\cdot|^{2}\phi_{R}}(x-x_i(t))\theta_{\epsilon}(t,x)dx.
\end{split}
\end{equation}
Using the integral kernel for $(-\Delta)^{\gamma/2}$, we have the pointwise identity
\begin{equation}
\label{eq:diss_term}
\begin{split}
&(-\Delta)^{\gamma/2}\paren*{|\cdot|^{2}\phi_{R}}(x-x_i(t)) = C_{\gamma}\PV\int_{\R^2}  \frac{\mathcal{H}_{\phi_R}(x-x_i(t),z)}{|z|^{2+\gamma}}dz, \qquad x\in\R^2,
\end{split}
\end{equation}
where
\begin{equation}
\mathcal{H}_{\phi_R}(x,z) \coloneqq 2|x|^{2}\phi_{R}(x) - |x-z|^{2}\phi_{R}(x-z)-|x+z|^{2}\phi_{R}(x+z), \qquad (x,z)\in\R^2\times\R^2.
\end{equation}
To estimate the modulus of \eqref{eq:diss_term}, we want to bound the $L^{\infty}$ norm of $(-\Delta)^{\gamma/2}(|\cdot-x_{i}(t)|^{2}\phi_{R}(\cdot-x_{i}(t)))$, as the $L^{1}$ norm of $\theta_{\epsilon}$ is $\ep$-independent. We use the triangle inequality to decompose the region of integration in \eqref{eq:diss_term} into the sub-regions $|z|\leq R/2$ and $|z|>R/2$.

\begin{itemize}[leftmargin=*]
\item
For $|z|\leq R/2$, Taylor's theorem yields that
\begin{equation}
\begin{split}
&2|x-x_{i}(t)|^{2}\phi_{R}(x-x_{i}(t))-|x-z-x_{i}(t)|^{2}\phi_{R}(x-z-x_{i}(t))-|x+z-x_{i}(t)|^{2}\phi_{R}(x+z-x_{i}(t)) \\
&\phantom{=}=\sum_{|\beta|=2}\frac{2}{\beta!}\int_{0}^{1}(1-s)\partial_{\beta}\paren*{|\cdot|^{2}\phi_{R}}(x-x_{i}(t) + sz) z^{\beta}ds,
\end{split}
\end{equation}
where the summation is over all multi-indices of order $2$. By the Leibnitz rule,
\begin{align}
|\partial_{\beta}\paren*{|\cdot|^{2}\phi_{R}}(x-x_{i}(t) + sz)| &\lesssim \|\phi_{R}\|_{L^{\infty}(\R^{2})} + R\|\nabla\phi_{R}\|_{L^{\infty}(\R^2)} + R^{2}\|\nabla^{2}\phi_{R}\|_{L^{\infty}(\R^{2})} \nonumber\\
&\leq \|\phi\|_{L^{\infty}(\R^2)} + \|\nabla\phi\|_{L^{\infty}(\R^2)} + \|\nabla^{2}\phi\|_{L^{\infty}(\R^2)}.
\end{align}
Therefore,
\begin{equation}
\left|\PV\int_{|z|\leq \frac{R}{2}} \frac{\mathcal{H}_{\phi_R}(x-x_i(t),z)}{|z|^{2+\gamma}}dz\right| \lesssim \int_{|z|\leq\frac{R}{2}} \frac{1}{|z|^{\gamma}}dz \lesssim_{\gamma} R^{2-\gamma},
\end{equation}
since $\gamma < 2$ by assumption.
\item
For $|z|>\frac{R}{2}$, we use the crude bound
\begin{equation}
\|\mathcal{H}_{\phi_R}\|_{L_{x,z}^\infty(\R^2\times\R^2)} \lesssim R^2\|\phi_R\|_{L^\infty(\R^2)} = R^2\|\phi\|_{L^\infty(\R^2)}
\end{equation}
to directly estimate
\begin{equation}
\left|\int_{|z|> \frac{R}{2}} \frac{\mathcal{H}_{\phi_R}(x-x_i(t),z)}{|z|^{2+\gamma}}dz\right| \lesssim R^{2}\|\phi_{R}\|_{L^{\infty}(\R^{2})} \int_{|z|>\frac{R}{2}} \frac{1}{|z|^{2+\gamma}}dz \lesssim_{\gamma} R^{2-\gamma},
\end{equation}
where the ultimate inequality follows from our assumption that $\gamma>0$.
\end{itemize}

Bookkeeping our estimates and summing over $i\in\{1,\ldots,N\}$, we have shown that
\begin{equation}
\label{eq:MI_G3}
|\mathrm{Group}_3| \lesssim_{\gamma} \kappa(\epsilon)N R^{2-\gamma}\int_{\R^2} \theta_{\epsilon}(t,x)dx = \kappa(\epsilon)N R^{2-\gamma}\|\ua_N\|_{\ell^1},
\end{equation}
which completes our estimation of $\mathrm{Group}_3$.

\subsection{Estimate for $\tl{\mu}_{R,\ep}$ and $\mu_{R,\ep}$}
To close the estimate for $\dot{I}_{R,\epsilon}$, it remains for us to estimate the auxiliary quantities $\tl{\mu}_{R,\ep}(t)$ and $\mu_{R,\epsilon}(t)$ (recall the definitions \eqref{eq:tlmu_def} and \eqref{eq:mu_def}, respectively) in terms of $I_{R,\epsilon}$, which we do with the next lemma.

\begin{lemma}[Estimate for $\tl{\mu}_{R,\ep}$, $\mu_{R,\ep}$]
\label{lem:mu_est}
For $0<\epsilon\ll R$, we have the pointwise estimate
\begin{equation}
\tl{\mu}_{R,\ep}(t) + \mu_{R,\ep}(t) \lesssim_{\alpha,\gamma} \int_{0}^{t} \paren*{\frac{\|\ua_N\|_{\ell^1}I_{R,\ep}(s)}{R^{4+\alpha}} + \frac{C_2^{\alpha/2}I_{R,\ep}(s)^{1+\frac{2-\alpha}{2}}}{\ep^\alpha R^{6-\alpha}}  + \frac{I_{R,\ep}(s)}{R^2 d_T(\ux_N^0)^{1+\alpha}} + \frac{\kappa(\ep)N\|\ua_N\|_{\ell^1}}{R^{\gamma}}}ds
\end{equation}
for all $t\in [0,T]$. 
\end{lemma}
\begin{proof}
Differentiating $\mu_{R,\epsilon}$ with respect to $t$ and using the chain rule, equations \eqref{eq:gSQG} and \eqref{eq:PVM}, and integration by parts, we find that
\begin{align}
\dot{\mu}_{R,\epsilon}(t) &= -\sum_{i=1}^N\sum_{1\leq j\neq i \leq N}a_{j}\int_{\R^2} \nabla\varphi_{R}(x-x_{i}(t))\cdot K_{\alpha}(x_{i}(t)-x_{j}(t)) \theta_{\epsilon}(t,x)dx \nn\\
&\ph{=} + \sum_{i=1}^N\int_{\R^{2}}\nabla\varphi_{R}(x-x_{i}(t))\cdot \frac{\nabla^{\perp}}{|\nabla|^{2-\alpha}}(\theta_{\epsilon})(t,x)\theta_{\epsilon}(t,x)dx \nn\\
&\ph{=} - \kappa(\ep)\sum_{i=1}^N \int_{\R^2}\varphi_R(x-x_i(t)) (-\Delta)^{\gamma/2}(\theta_\ep)(t,x)dx \nn\\
&\eqqcolon \mathrm{Term}_1 + \mathrm{Term}_2 + \mathrm{Term}_3.
\end{align}
Since $\nabla\varphi_R = -\nabla\phi_{R/4}$, it follows that
\begin{equation}
\begin{split}
\dot{\tl{\mu}}_{i,R,\ep}(t) &= \sum_{1\leq j\neq i \leq N} a_j\int_{\R^2}\nabla\varphi_{4R} \cdot K_\alpha(x_i(t)-x_j(t))\theta_\ep(t,x)dx \\
&\ph{=} + \int_{\R^2}\nabla\varphi_{4R}(x-x_i(t))\cdot \frac{\nabla^\perp}{|\nabla|^{2-\alpha}}(\theta_\ep)(t,x)\theta_\ep(t,x)dx \\
&\ph{=} + \kappa(\ep)\int_{\R^2}\phi_{R}(x-x_i(t)) (-\Delta)^{\gamma/2}(\theta_\ep)(t,x)dx.
\end{split}
\end{equation}
Since $\tl{\mu}_{i,R,\ep}(0) = 0$, it follows that to control $\tl{\mu}_{R,\ep}(t)$, it suffices to estimate $\mathrm{Term}_k$, for $k=1,2,3$, which we do now.

\subsubsection{Estimate for $\mathrm{Term}_2$}
We first estimate $\mathrm{Term}_2$. We use the integral kernel for the operator $\frac{\nabla^{\perp}}{|\nabla|^{2-\alpha}}$ to write
\begin{equation}
\begin{split}
&\int_{\R^2}\nabla\varphi_{R}(x-x_{i}(t))\cdot\frac{\nabla^{\perp}}{|\nabla|^{2-\alpha}}(\theta_{\epsilon})(t,x)\theta_{\epsilon}(t,x)dx \\
&\ph{=} = \PV\int_{\R^2\times\R^2} \theta_{\epsilon}(t,x)\theta_{\epsilon}(t,y) \nabla\varphi_{R}(x-x_{i}(t))\cdot K_{\alpha}(x-y)dxdy.
\end{split}
\end{equation}
Next, we decompose the region of integration $\R^{2}\times\R^{2}$ into the sub-regions $|x-y|\leq R/8$ and $|x-y|>R/8$ and consider each case separately.

\begin{itemize}[leftmargin=*]
\item
For $|x-y|>R/8$, we are localized away from the singularity of the kernel $K_{\alpha}$. Additionally, since $\varphi_{R}$ is supported in the region $|x|\geq R/4$ and identically $1$ in the region $|x|\geq R/2$, we see that $\supp( \nabla\varphi_{R})$ is contained in the annulus $R/2\geq |x|\geq R/4$. Therefore, we have the estimate
\begin{equation}
\begin{split}
&\left|\int_{|x-y|\geq R/8} \theta_{\epsilon}(t,x)\theta_{\epsilon}(t,y) \nabla\varphi_{R}(x-x_{i}(t))\cdot K_{\alpha}(x-y)dxdy\right| \\
&\ph{=}\lesssim_{\alpha} \frac{\|\ua_N\|_{\ell^1}}{R^{2+\alpha}} \paren*{\int_{\frac{R}{2}\geq |x-x_{i}(t)| \geq \frac{R}{4}} \theta_{\epsilon}(t,x)dx},
\label{eq:RHS_trick}
\end{split}
\end{equation}
where we also use the trivial size estimate $|1_{|x-y|>R/8}K_\alpha(x-y)|\lesssim R^{-1-\alpha}$ and $\|\theta(t)\|_{L^1(\R^2)}=\|\ua_N\|_{\ell^1}$. To bound the second factor in the RHS of \eqref{eq:RHS_trick} in terms $I_{R,\ep}(t)$, we use a trick from \cite{Marchioro90vNS}. Since $\phi$ is monotonically decreasing from $1$ to $0$ in the region $1\leq |x|\leq 2$ and strictly positive in the interior, we see that
\begin{align}
\int_{\frac{R}{2}\geq |x-x_{i}(t)| \geq \frac{R}{4}} \theta_{\epsilon}(t,x)dx &= \frac{1}{\tilde{\phi}_{R}(R/4)}\int_{\frac{R}{2}\geq |x-x_{i}(t)|\geq \frac{R}{4}} \tl{\phi}_{R}\paren*{\frac{R}{4}} \theta_{\epsilon}(t,x)dx \nn\\
&\leq \frac{1}{\tl{\phi}(1/4)}\int_{\frac{R}{2}\geq |x-x_{i}(t)|\geq \frac{R}{4}} \phi_{R}(x-x_{i}(t))\theta_{\epsilon}(t,x)dx \nn \\
&\lesssim \frac{1}{R^2}\int_{\R^{2}} |x-x_{i}(t)|^{2}\phi_{R}(x-x_{i}(t))\theta_{\epsilon}(t,x)dx, \label{eq:phi_mon}
\end{align}
where $\tl{\phi}_{R}(|x|) \coloneqq \tl{\phi}(|x|/R)$ (recall that $\phi=\tl{\phi}(|\cdot|)$). Trivially from the definition of $I_{R,\ep}(t)$, we have that
\begin{equation}
\sum_{i=1}^{N} \frac{1}{R^2}\int_{\R^{2}} |x-x_{i}(t)|^{2}\phi_{R}(x-x_{i}(t))\theta_{\epsilon}(t,x)dx = \frac{I_{R,\epsilon}(t)}{R^2}.
\end{equation}
\item
For $|x-y|\leq R/8$, we use the anti-symmetry of the Biot-Savart kernel to write
\begin{equation}
\begin{split}
&\PV\int_{|x-y|\leq \frac{R}{8}} \theta_{\epsilon}(t,x)\theta_{\epsilon}(t,y)\nabla\varphi_{R}(x-x_{i}(t))\cdot K_{\alpha}(x-y)dxdy \\
&\phantom{=} = \frac{1}{2}\PV\int_{|x-y|\leq \frac{R}{8}} \theta_{\epsilon}(t,x)\theta_{\epsilon}(t,y)K_{\alpha}(x-y)\cdot[\nabla\varphi_{R}(x-x_{i}(t))-\nabla\varphi_{R}(y-x_{i}(t))]dxdy.
\end{split}
\end{equation}
Next, note that the support of $\varphi_{R}$ implies that if $\varphi_{R}(x-x_{i}(t))\neq 0$ and $|x-y|\leq R/8$, then $|y-x_{i}(t)| \geq R/8$. We therefore obtain from the mean value theorem that
\begin{align}
&\frac{1}{2}\left|\PV \int_{|x-y|\leq \frac{R}{8}} \theta_{\epsilon}(t,x)\theta_{\epsilon}(t,y)K_{\alpha}(x-y)\cdot[\nabla\varphi_{R}(x-x_{i}(t))-\nabla\varphi_{R}(y-x_{i}(t))]dxdy\right| \nn\\
&\phantom{=} \lesssim_{\alpha} \frac{1}{R^2}\int_{\frac{R}{2}\geq |x-x_{i}(t)|\geq\frac{R}{4}} \theta_{\epsilon}(t,x)\paren*{\int_{\frac{5R}{8}\geq |y-x_{i}(t)|\geq \frac{R}{8}} \frac{\theta_{\epsilon}(t,y)}{|x-y|^{\alpha}}dy}dx \nn \\
&\phantom{=} \lesssim_{\alpha} \frac{\|\theta_{\epsilon}(t)\|_{L^{\infty}(\R^{2})}^{\alpha/2}}{R^2} \paren*{\int_{\frac{5R}{8} \geq |x-x_{i}(t)|\geq \frac{R}{8}} \theta_{\epsilon}(t,x)dx}^{1+\frac{2-\alpha}{2}}  \nn\\
&\phantom{=} \leq \frac{C_2^{\alpha/2}}{\epsilon^{\alpha}R^2}\paren*{\int_{\frac{5R}{8}\geq |x-x_{i}(t)|\geq \frac{R}{8}} \theta_{\epsilon}(t,x)dx}^{1+\frac{2-\alpha}{2}},
\end{align}
where we use \cref{prop:Linf_RP} with $(p,n,s)=(\infty,2,2-\alpha)$ and $\|\theta(t)\|_{L^1(\R^2)}=\|\ua_N\|_{\ell^1}$ to obtain the ultimate inequality. Since
\begin{equation}
\int_{\frac{5R}{8}\geq |x-x_{i}(t)|\geq \frac{R}{8}} \theta_{\epsilon}(t,x)dx \lesssim \frac{1}{R^2}\int_{\R^2} |x-x_{i}(t)|^{2}\phi_{R}(x-x_{i}(t))\theta_{\epsilon}(t,x)dx,
\end{equation}
we conclude from the embedding $\ell^1\subset \ell^{1+\frac{2-\alpha}{2}}$ that
\begin{equation}
\sum_{i=1}^{N} \frac{C_2^{\alpha/2}}{\epsilon^{\alpha}R^2}\paren*{\int_{\frac{5R}{8}\geq |x-x_{i}(t)|\geq \frac{R}{8}} \theta_{\epsilon}(t,x)dx}^{1+\frac{2-\alpha}{2}} \lesssim \frac{C_2^{\alpha/2}I_{R,\epsilon}(t)^{1+\frac{2-\alpha}{2}}}{\epsilon^{\alpha}R^{6-\alpha}} .
\end{equation}
\end{itemize}
Combining our estimates for both sub-regions, we obtain that
\begin{equation}
\label{eq:T2}
\mathrm{Term}_2 \lesssim_\alpha \frac{\|\ua_N\|_{\ell^1}I_{R,\ep}(t)}{R^{4+\alpha}} + \frac{C_2^{\alpha/2}I_{R,\ep}(t)^{1+\frac{2-\alpha}{2}}}{\ep^\alpha R^{6-\alpha}} .
\end{equation}

\subsubsection{Estimate for $\mathrm{Term}_1$}
We next estimate $\mathrm{Term}_1$. Using the trivial bound
\begin{equation}
\max_{1\leq i< j\leq N} \sup_{t\geq 0} |K_{\alpha}(x_i(t)-x_j(t))| \lesssim_\alpha \frac{1}{d_T(\ux_N^0)^{1+\alpha}},
\end{equation}
we can crudely estimate
\begin{align}
\left|\int_{\R^{2}} \nabla\varphi_{R}(x-x_{i}(t))\cdot K_{\alpha}(x_{i}(t)-x_{j}(t))\theta_{\epsilon}(t,x)dx\right| &\lesssim_{\alpha} \frac{1}{R d_{T}(\ux_{N}^{0})^{1+\alpha}}\int_{\frac{R}{2} \geq |x-x_{i}(t)|\geq \frac{R}{4}} \theta_{\epsilon}(t,x)dx \nn\\
&\lesssim \frac{1}{R^3 d_{T}(\ux_{N}^{0})^{1+\alpha}}\int_{\R^2} |x-x_{i}(t)|^{2}\phi_{R}(x-x_{i}(t))\theta_{\epsilon}(t,x)dx.
\end{align}
Summing over $i\in\{1,\ldots,N\}$, we obtain that
\begin{equation}
\label{eq:T1}
\mathrm{Term}_1 \lesssim_{\alpha} \frac{I_{R,\ep}(t)}{R^3 d_T(\ux_N^0)^{1+\alpha}} .
\end{equation}

\subsubsection{Estimate for $\mathrm{Term}_3$}
Lastly, we estimate $\mathrm{Term}_3$. By Plancherel's theorem and translation invariance of the fractional Laplacian, we have that
\begin{equation}
\int_{\R^2} \varphi_{R}(x-x_{i}(t)) (-\Delta)^{\gamma/2}(\theta_{\epsilon})(t,x)dx = \int_{\R^2} (-\Delta)^{\gamma/2}(\varphi_{R})(x-x_{i}(t)) \theta_{\epsilon}(t,x)dx.
\end{equation}
By H\"{o}lder's inequality, translation and dilation invariance of Lebesgue measure, and $\|\theta(t,\cdot)\|_{L^1(\R^2)}=\|\ua_N\|_{\ell^1}$, we have that
\begin{equation}
\left|\int_{\R^2} (-\Delta)^{\gamma/2}(\varphi_{R})(x-x_{i}(t)) \theta_{\epsilon}(t,x)dx\right| \lesssim \|(-\Delta)^{\gamma/2}\varphi_{R}\|_{L^{\infty}(\R^2)} \|\theta(t)\|_{L^{1}(\R^2)} \lesssim_{\gamma} \frac{\|\ua_N\|_{\ell^1}}{R^\gamma}.
\end{equation}
Summing over $i\in\{1,\ldots,N\}$, we conclude that
\begin{equation}
\label{eq:T3}
\mathrm{Term}_3 \lesssim_{\gamma} \frac{\kappa(\ep)N\|\ua_N\|_{\ell^1}}{R^{\gamma}}.
\end{equation}

\subsubsection{Bookkeeping}
Collecting our estimates \eqref{eq:T1}, \eqref{eq:T2}, and \eqref{eq:T3} for $\mathrm{Term}_1$, $\mathrm{Term}_2$, and $\mathrm{Term}_3$, respectively, we have shown that
\begin{equation}
\dot{\mu}_{R,\epsilon}(t) \lesssim_{\alpha,\gamma} \paren*{\frac{\|\ua_N\|_{\ell^1}I_{R,\ep}(t)}{R^{4+\alpha}} + \frac{C_2^{\alpha/2}I_{R,\ep}(t)^{1+\frac{2-\alpha}{2}}}{\ep^\alpha R^{6-\alpha}}  + \frac{I_{R,\ep}(t)}{R^3 d_T(\ux_N^0)^{1+\alpha}} + \frac{\kappa(\ep)N\|\ua_N\|_{\ell^1}}{R^{\gamma}}}.
\end{equation}
Provided that $\epsilon>0$ is sufficiently small depending on $R$, we have the inclusion $\supp(\theta_{i,\epsilon}^{0})\subset B(x_{i}^{0},R)$ for each $i\in \{1,\ldots,N\}$, which implies that $\mu_{R,\epsilon}(0)=0$. An application of the fundamental theorem of calculus completes the proof of \cref{lem:mu_est}.
\end{proof}

\subsection{Conclusion of Proof}
We now use the estimates for $\mu_{R,\ep}$ and $\tl{\mu}_{R,\ep}$ given by \cref{lem:mu_est} to close the estimate for $I_{R,\epsilon}(t)$. Collecting our estimates \eqref{eq:MI_G1}, \eqref{eq:MI_G2}, and \eqref{eq:MI_G3} for $\mathrm{Group}_1$, $\mathrm{Group}_2$, and $\mathrm{Group}_3$, respectively, we have that
\begin{equation}
\label{eq:con_sub}
\begin{split}
\dot{I}_{R,\ep}(t) &\lesssim_{\alpha,\gamma} \frac{\|\ua_N\|_{\ell^1}I_{R,\epsilon}(t)}{d_T(\ux_N^0)^{2+\alpha}}  + \frac{\|\ua_N\|_{\ell^1}\mu_{R,\ep}(t)}{R^\alpha} +  \frac{C_2^{\alpha/2}\mu_{R,\epsilon}(t)^{1+\frac{2-\alpha}{2}}}{\epsilon^{\alpha}} + \frac{R\|\ua_N\|_{\ell^1}\tl{\mu}_{R,\epsilon}(t)}{d_{T}(\ux_{N}^{0})^{1+\alpha}} \\
&\ph{=} +  \kappa(\epsilon)N R^{2-\gamma}\|\ua_N\|_{\ell^1}.
\end{split}
\end{equation}
As a consequence of \cref{lem:mu_est} and the nondecreasing property of $\bar{I}_{R,\ep}$, we obtain that
\begin{equation}\label{eq:mu_bnd}
\mu_{R,\epsilon}(\tau) + \tl{\mu}_{R,\ep}(\tau) \leq C(\alpha,\gamma)t\rho_{R,\ep}(t), \qquad 0\leq \tau \leq t\leq T,
\end{equation}
where $\rho_{R,\ep}: [0,T]\rightarrow [0,\infty)$ is the function defined by \eqref{eq:rho_def} and $\bar{I}_{R,\ep}$ is the maximal function defined in \eqref{eq:Imax_def}. Substituting the estimate \eqref{eq:mu_bnd} into the RHS of inequality \eqref{eq:con_sub} and integrating using the fundamental theorem of calculus, we find that
\begin{equation}
\begin{split}
I_{R,\epsilon}(\tau) &\leq I_{R,\ep}(0) + \kappa(\ep)N R^{2-\gamma}\|\ua_N\|_{\ell^1} t + \frac{\|\ua_N\|_{\ell^1}}{d_T(\ux_N^0)^{2+\alpha}}\int_0^t \bar{I}_{R,\ep}(s)ds \\
&\ph{=} + \paren*{\frac{C(\alpha,\gamma)\|\ua_N\|_{\ell^1}}{R^\alpha} + \frac{C(\alpha,\gamma)R\|\ua_N\|_{\ell^1}}{d_T(\ux_N^0)^{1+\alpha}}}\int_0^t s\rho_{R,\ep}(s)ds \\
&\ph{=} + \frac{C(\alpha,\gamma)^{1+\frac{2-\alpha}{2}} C_2^{\alpha/2}}{\ep^\alpha} \int_0^t s^{1+\frac{2-\alpha}{2}}\rho_{R,\ep}(s)^{1+\frac{2-\alpha}{2}}ds.
\end{split}
\end{equation}
for all $\tau\in [0,t]$. Taking the supremum of the LHS over $\tau \in [0,t]$ completes the proof of \cref{prop:Ibar}.

\section{Bootstrapped Moment of Inertia Estimate}
\label{sec:boot}
We cannot directly apply \cref{prop:GB} to the estimate given by \cref{prop:Ibar} due to the fact that we have powers of $\bar{I}_{R,\ep}$ with exponent strictly greater than one appearing in the RHS. If we knew, for instance, that $\bar{I}_{R,\ep}(t) \lesssim \ep^2$ for $t\in [0,T]$, then for any $\beta>1$, we could use the trivial estimate $\bar{I}_{R,\ep}(t)^\beta \lesssim_\beta \epsilon^{2(\beta-1)}\bar{I}_{R,\ep}(t)$; however, such an inequality for $\bar{I}_{R,\ep}$ is precisely what we are trying to prove. Instead, we must proceed by a bootstrap argument, first considering $\bar{I}_{R,\ep}(t)$ on a subinterval $[0,T_*]\subset [0,T]$ on which we know that $\bar{I}_{R,\ep}(t) \leq D\ep^2$, for some given constant $D>1$. We then use the estimate of \cref{prop:Ibar} together with \cref{prop:GB} in order to obtain a positive lower bound for $T_*$ depending on $D$ and some other parameters, but crucially not on $\ep$. Thus, the goal of this section is to prove the following proposition.

\begin{prop}
\label{prop:b_MI}
There exist constants $\ep_0>0$ and $D>1$ and a time $T_*\in (0,T]$ all depending on the data $(\alpha,\gamma,\kappa,C_1,C_2,d_0(\ux_N^0),\|\ua_N\|_{\ell^1}, N, T,R)$ such that
\begin{equation}
\bar{I}_{R,\ep}(T_*) \leq D\max\{\kappa(\ep),\ep^2\}
\end{equation}
for all $\ep\in (0,\ep_0]$. Moreover, if $\alpha\in [0,1)$, then $T_*=T$.
\end{prop}

Using \cref{prop:b_MI} together with \cref{lem:mu_est}, we easily obtain the following corollary.
\begin{cor}
\label{cor:ai_diff}
There exist constants $C>1$ and $\ep_0>0$ and a time $T_*\in [0,T]$ all depending on the data $(\alpha,\gamma,\kappa, C_1, C_2,d_0(\ux_N^0), \|\ua_N\|_{\ell^1},N,T,R)$ such that
\begin{equation}
\max_{i=1,\ldots,N} \sup_{t\in [0,T_*]} \left|a_i - \int_{|x-x_i(t)|\leq R}\theta_\ep(t,x)dx\right| + \left|\int_{|x-x_i(t)|\geq R}\theta_\ep(t,x)dx\right| \leq C\max\{\kappa(\ep),\ep^2\}
\end{equation}
for every $0<\ep\leq \ep_0$. Moreover, if $\alpha \in [0,1)$, then $T_*=T$.
\end{cor}

\begin{remark}
The dependence of $T_*$ on the parameter $R$ for $\alpha=1$ in the statements of \cref{prop:b_MI} and \cref{cor:ai_diff} prevents us from proving localization of the SQG flow even for short times.
\end{remark}

To prove \cref{prop:b_MI}, we first show that the maximal function $\bar{I}_{R,\ep}$ is a continuous function of time, so that $\bar{I}_{R,\ep}$ satisfies the assumptions of \cref{prop:GB}.

\begin{lemma}[Continuity of $\bar{I}_{R,\ep}$]
\label{lem:I_cont}
For any interval $[0,T]$ contained in the lifespan of $\theta_\ep$, we have that $\bar{I}_{R,\epsilon}:[0,T] \rightarrow \R$ is continuous.
\end{lemma}
\begin{proof}
First, it is evident from the continuity properties of $\ux_N$ and $\theta_\ep$ that $I_{R,\ep}$ is a continuous function of time. Moreover, it is tautological that $\bar{I}_{R,\ep}$ is nondecreasing in time. Let $t\in [0,T]$ and let $\delta \in\R$ with $|\delta|\ll t$.

Suppose first that $\delta>0$. Since $[0,\min\{t+\delta,T\}]$ is compact and $I_{R,\ep}$ is continuous, there exists a point $t_* \in [0,\min\{t+\delta,T\}]$ such that $I_{R,\ep}(t_*) = \bar{I}_{R,\ep}(\min\{t+\delta,T\})$. We consider two cases.
\begin{itemize}
\item
If $t_{*}\in[0,t]$, then since $\bar{I}_{R,\ep}$ is nondecreasing, it follows that
\begin{equation}
\bar{I}_{R,\ep}(t) \geq I_{R,\ep}(t_{*})=\bar{I}_{R,\ep}(\min\{t+\delta,T\}) \geq \bar{I}_{R,\ep}(t),
\end{equation}
which implies that $\bar{I}_{R,\ep}(\min\{t+\delta,T\}) = \bar{I}_{R,\ep}(t)$ and therefore $\bar{I}_{R,\ep}(s) = \bar{I}_{R,\ep}(t)$ for any $s\in [t,\min\{t+\delta,T\}]$.
\item
If $t_{*}\in (t,\min\{t+\delta,T\}]$, then since $I_{R,\ep}$ is continuous, given $\varepsilon>0$, there exists $\delta'>0$ such that $|t'-t| \leq \delta'$ implies that $|I_{R,\ep}(t')-I_{R,\ep}(t)| < \varepsilon$. Provided that $\delta\leq \delta'$, we then have that $|I_{R,\ep}(t_*)-I_{R,\ep}(t)| < \varepsilon$. Hence,
\begin{equation}
\bar{I}_{R,\ep}(t) + \varepsilon \geq I_{R,\ep}(t) + \varepsilon > I_{R,\ep}(t_{*}) = \bar{I}_{R,\ep}(\min\{t+\delta,T\}),
\end{equation}
which implies that for any $s \in [t,\min\{t+\delta,T\}]$, $\bar{I}_{R,\ep}(\min\{t+\delta,T\})-\bar{I}_{R,\ep}(s) < \varepsilon$.
\end{itemize}
Suppose now that $\delta<0$. Since $[0,\max\{0,t+\delta\}]$ is compact and $I_{R,\ep}$ is continuous, there exists $t_*\in [0,t]$ such that $I_{R,\ep}(t_*) = \bar{I}_{R,\ep}(t)$. As before, we consider two cases.
\begin{itemize}
\item
If $t_* \in [0,\max\{t+\delta,0\}]$, then since $\bar{I}_{R,\ep}$ is nondecreasing, it follows that
\begin{equation}
\bar{I}_{R,\ep}(t) \geq \bar{I}_{R,\ep}(\max\{t+\delta,0\}) \geq {I}_{R,\ep}(t_*) = \bar{I}_{R,\ep}(t),
\end{equation}
which implies that $\bar{I}_{R,\ep}(t) = \bar{I}_{R,\ep}(\max\{t+\delta,0\})$. Hence, $\bar{I}_{R,\ep}(s) = \bar{I}_{R,\ep}(t)$ for all $s\in [\max\{t+\delta,0\},t]$.
\item
If $t_* \in (\max\{t+\delta,0\},t]$, then using the same continuity argument as before, we have that for given $\varepsilon>0$,
\begin{equation}
\bar{I}_{R,\ep}(\max\{t+\delta,0\}) + \varepsilon \geq I_{R,\ep}(\max\{t+\delta,0\}) + \varepsilon > I_{R,\ep}(t_*)=\bar{I}_{R,\ep}(t),
\end{equation}
which implies that $\bar{I}_{R,\ep}(t) - \bar{I}_{R,\epsilon}(s) < \varepsilon$ for all $s\in  [\max\{t+\delta,0\},t]$, provided that $|\delta| = |\delta(\varepsilon)|$ is sufficiently small.
\end{itemize}
Thus, we have shown that for all $\varepsilon>0$, there exists $\delta=\delta(\varepsilon)>0$ such that $|s-t|\leq \delta$ implies that $|\bar{I}_{R,\ep}(t)-\bar{I}_{R,\ep}(s)| < \varepsilon$. Since $t\in [0,T]$ was arbitrary, we are done.
\end{proof}

We now proceed with the proof of \cref{prop:b_MI}. Let $D\gg 1$ be a given constant. By \cref{lem:I_cont} and the compactness of the interval $[0,T]$, there exists a time $T_{*}\in[0,T]$ such that $\bar{I}_{R,\epsilon}(T_{*})=D\max\{\ep^2,\kappa(\ep)\}$. Moreover, since
\begin{equation}
\bar{I}_{R,\ep}(0) = I_{R,\ep}(0) = \sum_{i=1}^N \int_{\R^2}|x-x_i^0|^2\theta_{i,\ep}^0(x)dx \leq C_1^2\ep^2 \|\ua_N\|_{\ell^1},
\end{equation}
it follows that $T_*>0$ provided that $D> C_1^2\|\ua_N\|_{\ell^1}$. Now for $t\in [0,T_*]$, we have the bound
\begin{equation}
\label{eq:rho_bnd1}
\begin{split}
\rho_{R,\ep}(t) &\leq \paren*{\frac{\|\ua_N\|_{\ell^1}}{R^{4+\alpha}} + \frac{C_2^{\alpha/2}D^{\frac{2-\alpha}{2}} \max\{\frac{\kappa(\ep)^{\frac{2-\alpha}{2}}}{\ep^\alpha},\ep^{2(1-\alpha)}\}}{R^{6-\alpha}} + \frac{1}{R^3 d_T(\ux_N^0)^{1+\alpha}}} \bar{I}_{R,\ep}(t) \\
&\ph{=} + \frac{\kappa(\ep)N\|\ua_N\|_{\ell^1}}{R^\gamma}.
\end{split}
\end{equation}
Similarly, using the convexity of the function $z\mapsto z^{1+\frac{2-\alpha}{2}}$, we have the bound
\begin{equation}
\label{eq:rho_bnd2}
\begin{split}
\rho_{R,\ep}(t)^{1+\frac{2-\alpha}{2}} &\lesssim_\alpha \paren*{\frac{\|\ua_N\|_{\ell^1}}{R^{4+\alpha}} + \frac{C_2^{\alpha/2}D^{\frac{2-\alpha}{2}} \max\{\frac{\kappa(\ep)^{\frac{2-\alpha}{2}}}{\ep^\alpha},\ep^{2(1-\alpha)}\}}{R^{6-\alpha}} + \frac{1}{R^3 d_T(\ux_N^0)^{1+\alpha}}}^{\frac{4-\alpha}{2}}  \\
&\ph{=}\qquad \times D^{\frac{2-\alpha}{2}}\max\{\kappa(\ep)^{\frac{2-\alpha}{2}}, \ep^{2-\alpha}\}\bar{I}_{R,\ep}(t) \\
&\ph{=} + \paren*{\frac{\kappa(\ep)N\|\ua_N\|_{\ell^1}}{R^\gamma}}^{\frac{4-\alpha}{2}},
\end{split}
\end{equation}
for $t\in [0,T_*]$. Using the estimates \eqref{eq:rho_bnd1} and \eqref{eq:rho_bnd2} to majorize the RHS of the inequality \eqref{eq:I_ineq} given by \cref{prop:Ibar} and then applying \cref{prop:GB}, we obtain that
\begin{equation}
\label{eq:sig_ineq}
\bar{I}_{R,\epsilon}(t) \leq \sigma_{1,\ep,D}(t)\exp\paren*{\sigma_{2,\ep,D}(t)}, \qquad t\in [0,T_*],
\end{equation}
where $\sigma_{1,\ep,D}, \sigma_{2,\ep,D}: [0,T]\rightarrow [0,\infty)$ are the functions respectively defined by
\begin{equation}
\label{eq:sig1}
\begin{split}
\sigma_{1,\ep,D}(t) &\coloneqq I_{R,\ep}(0) + \kappa(\ep)NR^{2-\gamma}\|\ua_N\|_{\ell^1}t + \frac{C'(\alpha,\gamma)\|\ua_N\|_{\ell^1}^2 \kappa(\ep)N}{R^{\alpha+\gamma}}t^2 \\
&\ph{=} + \frac{C_2^{\alpha/2}C'(\alpha,\gamma) (\kappa(\ep)N\|\ua_N\|_{\ell^1})^{\frac{4-\alpha}{2}}}{\ep^\alpha R^{\frac{(4-\alpha)\gamma}{2}}} t^{\frac{6-\alpha}{2}} + \frac{C'(\alpha,\gamma)\kappa(\ep)N\|\ua_N\|_{\ell^1}^2}{d_T(\ux_N^0)^{1+\alpha}R^\gamma}t^2
\end{split}
\end{equation}
and
\begin{equation}
\label{eq:sig2}
\begin{split}
\sigma_{2,\ep,D}(t) &\coloneqq \frac{\|\ua_N\|_{\ell^1}}{d_T(\ux_N^0)^{2+\alpha}} \\
&\ph{=} + \frac{C'(\alpha,\gamma)\|\ua_N\|_{\ell^1}t^2}{R^\alpha}\paren*{\frac{\|\ua_N\|_{\ell^1}}{R^{4+\alpha}} + \frac{C_2^{\alpha/2}D^{\frac{2-\alpha}{2}}\max\{\frac{\kappa(\ep)^{\frac{2-\alpha}{2}}}{\ep^\alpha},\ep^{2(1-\alpha)}\}}{R^{6-\alpha}} + \frac{1}{R^3 d_T(\ux_N^0)^{1+\alpha}}} \\
&\ph{=} + C_2^{\alpha/2}C'(\alpha,\gamma) D^{\frac{2-\alpha}{2}}\max\{\frac{\kappa(\ep)^{\frac{2-\alpha}{2}}}{\ep^\alpha},\ep^{2(1-\alpha)}\} t^{\frac{6-\alpha}{2}} \\
&\ph{=}\qquad \times \paren*{\frac{\|\ua_N\|_{\ell^1}}{R^{4+\alpha}} + \frac{C_2^{\alpha/2}D^{\frac{2-\alpha}{2}}\max\{\frac{\kappa(\ep)^{\frac{2-\alpha}{2}}}{\ep^\alpha},\ep^{2(1-\alpha)}\}}{R^{6-\alpha}} + \frac{1}{R^3 d_T(\ux_N^0)^{1+\alpha}}}^{\frac{4-\alpha}{2}} \\
&\ph{=} + \frac{C'(\alpha,\gamma) \|\ua_N\|_{\ell^1} t^2}{d_T(\ux_N^0)^{1+\alpha}} \paren*{\frac{\|\ua_N\|_{\ell^1}}{R^{4+\alpha}} + \frac{C_2^{\alpha/2}D^{\frac{2-\alpha}{2}}\max\{\frac{\kappa(\ep)^{\frac{2-\alpha}{2}}}{\ep^\alpha},\ep^{2(1-\alpha)}\}}{R^{6-\alpha}} + \frac{1}{R^3 d_T(\ux_N^0)^{1+\alpha}}}.
\end{split}
\end{equation}
Suppose that $T_*<T$. Evaluating inequality \eqref{eq:sig_ineq} at $t=T_*$, we find that
\begin{equation}
\label{eq:T*_lb}
T_*\geq \frac{1}{\sigma_{2,\ep,D}(T)} \ln\paren*{\frac{D\max\{\kappa(\ep),\ep^2\}}{\sigma_{1,\ep,D}(T)}}.
\end{equation}
We need $\max\{\kappa(\ep),\ep^2\}/\sigma_{1,\ep,D}(T)$ to be bounded from below as $\ep\rightarrow 0^+$ in order for the lower bound \eqref{eq:T*_lb} to be useful. Since $\kappa(\ep)^{\frac{2-\alpha}{2}}/\ep^\alpha\rightarrow 0$ as $\ep\rightarrow  0^+$ by assumption, it follows that
\begin{align}
\frac{\max\{\kappa(\ep),\ep^2\}}{\sigma_{1,\ep,D}(T)} &\geq \Big(C_1^2\|\ua_N\|_{\ell^1}+C_\kappa NR^{2-\gamma}\|\ua_N\|_{\ell^1} T + \frac{C'(\alpha,\gamma)C_\kappa N\|\ua_N\|_{\ell^1}^2 T^2}{R^{\alpha+\gamma}}  \nn \\
&\ph{=} \qquad + \frac{C'(\alpha,\gamma) C_2^{\alpha/2} C_\kappa(N\|\ua_N\|_{\ell^1})^{\frac{4-\alpha}{2}}\kappa(\ep) T^{\frac{6-\alpha}{2}}}{ R^{\frac{(4-\alpha)\gamma}{2}}} \nn\\
&\ph{=} \qquad + \frac{C'(\alpha,\gamma)C_\kappa N \|\ua_N\|_{\ell^1}^2 T^2}{d_T(\ux_N^0)^{1+\alpha} R^{\gamma}} \Big)^{-1} \nn\\
&\eqqcolon \beta_1.
\end{align}
Note that $\beta_1$ is independent of $\ep$. Hence,
\begin{equation}
\label{eq:T*_beta}
T_*\geq \frac{\ln(D\beta_1)}{\sigma_{2,\ep,D}(T)}.
\end{equation}
We now divide into two cases based on the value of the parameter $\alpha$: $\alpha \in [0,1)$ and $\alpha=1$.

\begin{itemize}[leftmargin=*]
\item
If $\alpha\in[0,1)$, then an examination of the definition \eqref{eq:sig2} of $\sigma_{2,\ep,D}$ then shows that
\begin{align}
\lim_{\ep\rightarrow 0^+} \sigma_{2,\ep,D}(T) &= \frac{\|\ua_N\|_{\ell^1}}{d_T(\ux_N^0)^{2+\alpha}} + C'(\alpha,\gamma)\|\ua_N\|_{\ell^1}T^2\paren*{\frac{1}{R^\alpha}+\frac{1}{d_T(\ux_N^0)^{1+\alpha}}}\paren*{\frac{\|\ua_N\|_{\ell^1}}{R^{4+\alpha}}+\frac{1}{R^3 d_T(\ux_N^0)^{1+\alpha}}} \nn\\
&\eqqcolon \beta_2.
\end{align}
Thus, first choosing $D>1$ sufficiently large so that
\begin{equation}
\ln(D\beta_1) = 3\beta_2 T,
\end{equation}
then choosing
\begin{equation}
\ep_0 = \ep(\alpha,\gamma,C_\kappa,C_1,C_2,\|\ua_N\|_{\ell^1}, D,d_T(\ux_N^0),T)>0
\end{equation}
sufficiently small so that $\sigma_{2,\ep,D}(T) \leq 2\beta_2$ for all $\ep\in (0,\ep_0]$, we obtain from \eqref{eq:T*_beta} the inequality
\begin{equation}
T_* \geq \frac{3 \beta_2 T }{2\beta_2} = \frac{3T}{2},
\end{equation}
which contradicts our assumption that $T_*<T$. Therefore, we conclude that
\begin{equation}
\bar{I}_{R,\ep}(T) \leq \frac{e^{\frac{3 \beta_2 T}{2}} \max\{\kappa(\ep),\ep^2\}}{\beta_1}, \qquad 0<\ep\leq \ep_0.
\end{equation}
\item
If $\alpha=1$, then we cannot obtain a $D$-independent upper bound for $\sigma_{2,\ep,D}$ by choosing $\ep$ sufficiently small, as we did in the case $\alpha\in (0,1]$, since we only have the upper bound
\begin{equation}
\begin{split}
\sigma_{2,\ep,D}(T) &\leq \frac{\|\ua_N\|_{\ell^1}}{d_T(\ux_N^0)^{3}} \\
&\ph{=} + C'(1,\gamma)\|\ua_N\|_{\ell^1}T^2\paren*{\frac{1}{R} + \frac{1}{d_T(\ux_N^0)^{2}}}\paren*{\frac{\|\ua_N\|_{\ell^1}}{R^{5}} + \frac{C_2^{1/2} D^{1/2}}{R^{3}} + \frac{1}{R^2 d_T(\ux_N^0)^{2}}} \\
&\ph{=} + C_2^{1/2}C'(1,\gamma)D^{1/2} T^{5/2}\paren*{\frac{\|\ua_N\|_{\ell^1}}{R^{5}} + \frac{C_2^{1/2} D^{1/2}}{R^{3}} + \frac{1}{R^2 d_T(\ux_N^0)^{2}}}.
\end{split}
\end{equation}
However, by choosing $D>0$ sufficiently large so that $D\beta_1>1$, we obtain a positive lower bound for $T_*$ which is independent of $\ep$. While this lower bound deteriorates as $D\rightarrow \infty$, we can nevertheless optimize the choice of $D$ to conclude that there exists a tuple of positive constants $(\ep_0, T_*, D_*)$ depending on the data
\begin{equation}
(\alpha,\gamma,\kappa,C_1,C_2, \|\ua_N\|_{\ell^1}, d_T(\ux_N^0), N,T,R)
\end{equation}
such that
\begin{equation}
\bar{I}_{R,\ep}(T_*) \leq D_*\max\{\kappa(\ep),\ep^2\}, \qquad 0<\ep\leq \ep_0,
\end{equation}
which completes the proof of \cref{prop:b_MI}.
\end{itemize}

\section{Proof of \cref{thm:main_mSQG}}
\label{sec:pf_main}
We now have all the ingredients necessary to prove weak-* convergence of the Borel measure $\theta_{\epsilon}(t)$ to the empirical density function $\sum_{i=1}^{N}a_{i}\delta_{x_{i}(t)}$ as $\ep\rightarrow 0^+$ uniformly on $[0,T]$, for $\alpha \in [0,1)$. Let $f\in C(\R^{2})$ be a test function. Let $0<R\leq d_T(\ux_N^0)/100$ be a parameter, the precise value of which will be specified momentarily. By the triangle inequality
\begin{align}
\left|\int_{\R^2}f(x)[\theta_{\epsilon}(t,x)-\sum_{i=1}^{N}a_{i}\delta_{x_{i}(t)}]dx\right| &\leq \left|\sum_{i=1}^{N}\paren*{\int_{|x-x_{i}(t)|\leq R}\theta_{\epsilon}(t,x)f(x)dx-a_{i}f(x_{i}(t))}\right| \nn\\
&\ph{=} + \sum_{i=1}^{N}\int_{|x-x_{i}(t)|>R} \theta_{\epsilon}(t,x)f(x)dx \nn\\
& \eqqcolon \mathrm{Term}_{1,\ep}(t) +\mathrm{Term}_{2,\ep}(t). 
\end{align}
We estimate $\mathrm{Term}_{1,\ep}$ and $\mathrm{Term}_{2,\ep}$ separately.

\begin{description}[leftmargin=*]
\item[Estimate for $\mathrm{Term}_{2,\ep}$:]
Let $\ep_0(R)>0$ and $C(R)>1$ be the constants given by \cref{cor:ai_diff}. Then by direct majorization of the integrand of $\mathrm{Term}_{2,\ep}$ together with the boundedness of $f$, it follows immediately that
\begin{equation}
\label{eq:m_fin_T2}
\sup_{t\in [0,T]} \mathrm{Term}_{2,\ep}(t) \leq C(R)\max\{\kappa(\ep),\epsilon^{2}\}\|f\|_{L^{\infty}(\R^2)}
\end{equation}
for all $0<\ep\leq \ep_0(R)$. It follows trivially that for $R$ fixed, $\mathrm{Term}_{2,\ep}\rightarrow 0$ as $\ep\rightarrow 0^+$ uniformly on $[0,T]$.

\item[Estimate for $\mathrm{Term}_{1,\ep}$:]
Since $N$ is finite, the trajectories $x_i(t)$ are continuous, and the interval $[0,T]$ is compact, it follows that the point vortices $x_{i}(t)$ are confined to a compact region $K\subset \R^{2}$ on the interval $[0,T]$. \footnote{By conservation of the point vortex moment of inertia $I(t)$ and the fact the intensities $a_1,\ldots,a_N$ have the same sign, we actually have that $x_i(t)$ are confined to a compact region of $\R^2$ globally in time.} Since $f$ is continuous, $f$ is uniformly continuous on $K$. Hence, given any $\varepsilon>0$,
\begin{equation}
\max_{i=1,\ldots,N}\sup_{t\in [0,T]}\sup_{|x-x_{i}(t)|\leq R}|f(x)-f(x_{i}(t))| \leq \varepsilon
\end{equation}
provided that we choose $R=R(\varepsilon)$ sufficiently small. Therefore, by the triangle inequality,
\begin{align}
\mathrm{Term}_{1,\ep}(t) &\leq \varepsilon\sum_{i=1}^{N}\int_{|x-x_{i}(t)| \leq R}\theta_{\epsilon}(t,x)dx + \sum_{i=1}^{N}|f(x_{i}(t))|\left|\int_{|x-x_{i}(t)|\leq R}\theta_{\epsilon}(t,x)dx-a_{i}\right| \nn\\
&\eqqcolon \mathrm{Term}_{1,1,\ep}(t) + \mathrm{Term}_{1,2,\ep}(t).
\end{align}
For $\mathrm{Term}_{1,1,\ep}$, we have the crude bound
\begin{equation}
\mathrm{Term}_{1,1,\ep}(t) \leq \varepsilon\|\theta_\ep(t)\|_{L^1(\R^2)} = \varepsilon\|\ua_N\|_{\ell^1}.
\end{equation}
For $\mathrm{Term}_{1,2,\ep}$, we use \cref{cor:ai_diff} to obtain that
\begin{equation}
\sup_{t\in [0,T]} \sum_{i=1}^N \left|\int_{|x-x_i(t)|\leq R}\theta_\ep(t,x)dx-a_i\right| \leq C(R)\max\{\kappa(\ep),\ep^2\}
\end{equation}
for all $0<\ep\leq \ep_0(R)$, for constants $C(R)$ and $\ep_0(R)$ as above. Hence,
\begin{equation}
\sup_{t\in [0,T]} \mathrm{Term}_{1,2,\ep}(t) \leq C(R)\|f\|_{L^\infty(\R^2)}\max\{\kappa(\ep),\ep^2\}
\end{equation}
for all $0<\ep\leq \ep_0(R)$. Combining our estimates for $\mathrm{Term}_{1,1,\ep}$ and $\mathrm{Term}_{1,2,\ep}$, we see that
\begin{equation}
\sup_{t\in [0,T]} \mathrm{Term}_{1,\ep}(t) \leq \varepsilon\|\ua_N\|_{\ell^1} + C(R)\|f\|_{L^\infty}\max\{\kappa(\ep),\ep^2\}
\end{equation}
for all $0<\ep\leq \ep_0(R)$. Evidently, the RHS of the preceding inequality tends to $\varepsilon\|\ua_N\|_{\ell^1}$ as $\ep\rightarrow 0^+$. Since $\varepsilon>0$ was arbitrary, we conclude that
\begin{equation}
\label{eq:m_fin_T1}
\limsup_{\ep\rightarrow 0^+} \sup_{t\in [0,T]} \mathrm{Term}_{1,\ep}(t)=0.
\end{equation}
\end{description}

Collecting our estimates \eqref{eq:m_fin_T1} and \eqref{eq:m_fin_T2} for $\mathrm{Term}_{1,\ep}$ and $\mathrm{Term}_{2,\ep}$, respectively, we conclude that
\begin{equation}
\limsup_{\ep\rightarrow 0^+} \sup_{t\in [0,T]} \left|\int_{\R^2}f(x)[\theta_\ep(t,x)-\sum_{i=1}^N a_i\delta_{x_i(t)}]dx\right| = 0,
\end{equation}
as desired.

\bibliographystyle{siam}
\bibliography{PointVortex}
\end{document}